\documentclass[12pt]{amsart}
\usepackage{mathscinet}
\usepackage{amsmath,amsthm,amsfonts,amssymb,amscd,euscript}
\usepackage{textcomp}
\usepackage{color}
\usepackage{graphics}
\usepackage{tikz}
\usetikzlibrary{plotmarks}
\input{xy}
\xyoption{all}
\usepackage[colorlinks=true,pagebackref,hyperindex]{hyperref}
\usepackage{cleveref}

\newlength{\defbaselineskip}
\theoremstyle{plain}
\newtheorem{theorem}{Theorem}[section]

\newtheorem{proposition}[theorem]{Proposition}
\newtheorem{corollary}[theorem]{Corollary}
\newtheorem{lemma}[theorem]{Lemma}

\theoremstyle{definition}
\newtheorem{example}[theorem]{Example}
\newtheorem{definition}[theorem]{Definition}

\newtheorem{remark}[theorem]{Remark}

\newtheorem{conjecture}[theorem]{Conjecture}

\newcommand{\supp}{\operatorname{supp}}

\newcommand{\Aut}{\operatorname{Aut}}

\newcommand{\Rect}{\operatorname{Rect}}

\newcommand{\aaa}{\bar{a}}
\newcommand{\bbb}{\bar{b}}

\newcommand{\R}{\mathbb{R}}
\newcommand{\RRR}{\mathcal{R}}
\newcommand{\C}{\mathbb{C}}

\newcommand{\T}{\mathbb{T}}
\newcommand{\Z}{\mathbb{Z}}

\newcommand{\ff}{F}
\newcommand{\g}{{\alpha}}
\newcommand{\GG}{G}
\newcommand{\EEE}{W_F}

\newcommand{\ZZZ}{Z}

\theoremstyle{plain}
\newtheorem{theoremalpha}{Theorem}

\newtheorem{conjecturealpha}[theoremalpha]{Conjecture}

\theoremstyle{plain}

\theoremstyle{definition}

\numberwithin{equation}{section}

\newtheorem*{rep@theorem}{\rep@title} \newcommand{\newreptheorem}[2]{%
\newenvironment{rep#1}[1]{%
\def\rep@title{\bf #2 \ref{##1}}%
\begin{rep@theorem} }%
{\end{rep@theorem} } }
\makeatother
\newreptheorem{theorem}{Theorem}
\newreptheorem{conjecture}{Conjecture}

\begin{document}
\title[The two-dimensional Jacobian conjecture]{On the two-dimensional Jacobian conjecture: Magnus' formula revisited, IV}

\author{Kyungyong Lee}
\address{Department of Mathematics, University of Alabama,
	Tuscaloosa, AL 35487, U.S.A. 
	and Korea Institute for Advanced Study, Seoul 02455, Republic of Korea}
\email{kyungyong.lee@ua.edu; klee1@kias.re.kr}

\author{Li Li}
\address{Department of Mathematics and Statistics,
Oakland University, 
Rochester, MI 48309, U.S.A.}
\email{li2345@oakland.edu}

\begin{abstract}
Let $(F,G)$ be a Jacobian pair with $d=w\text{-deg}(F)$ and $e=w\text{-deg}(G)$ for some direction $w$.  
A generalized Magnus' formula approximates $G$ as  
$\sum_{\gamma\ge 0} c_\gamma F^{\frac{e-\gamma}{d}}$
for some complex numbers $c_\gamma$. 
We develop an approach to the two-dimensional Jacobian conjecture, aiming to minimize the use of terms corresponding to $\gamma>0$.
As an initial step in this approach, 
we define and study the inner polynomials
of $F$ and $G$. 
The main result of this paper shows that the northeastern vertex of the Newton polygon of each inner polynomial is located within a specific region. As applications of this result, we introduce several conjectures and  prove some of them for special cases.
\end{abstract}

\thanks{KL was supported by the University of Alabama, Korea Institute for Advanced Study, and the NSF grant DMS 2042786 and DMS 2302620.}

\maketitle
\setcounter{tocdepth}{1}
\tableofcontents

\section{Introduction}\label{section_intro}

The Jacobian conjecture, proposed by Keller \cite{Keller}, has been studied by numerous mathematicians. A partial list of related works includes \cite{AM,A,AdEs,ApOn,BCW,BK,DEZ,dBY,NVC,CW1,CW2,CZ,Dru,EssenTutaj,EssenWZ,Gwo,Hub,JZ,Kire,LM,M,MU,MO,Nagata,NaBa,Wang,Yag,Yu}.  Comprehensive surveys can be found in \cite{Essen,vdEssen}. In this series of papers, we focus exclusively on the plane case; therefore, whenever we refer to the Jacobian conjecture, we mean the two-dimensional Jacobian conjecture.

Let $\mathcal{R}=\mathbb{C}[x,y]$. For simplicity, denote $[f, g] :=\det \begin{pmatrix}{\partial f}/{\partial x} & {\partial g}/{\partial x}\\ {\partial f}/{\partial y} &{\partial g}/{\partial y}\end{pmatrix}\in \mathcal{R}$ for any pair of polynomials $f,g\in \mathcal{R}$.

\noindent \textbf{Jacobian conjecture.} 
\emph{Let $f,g\in \mathcal{R}$. 
Consider the endomorphism }$\pi: \mathcal{R}\longrightarrow \mathcal{R}$\emph{ given by }
$\pi(x)=f$\emph{ and }$\pi(y)=g$. \emph{If }$[f,g]\in \mathbb{C}\setminus\{0\}$ (in which case $(f,g)$ is called a \textit{Jacobian pair}), \emph{then }$\pi$\emph{ is bijective.}

The following result, due to Bialynicki-Birula and Rosenlicht, asserts that the term ``bijective'' in the Jacobian conjecture can be replaced with  ``polynomial automorphism''. 
\begin{theorem}\cite{Rosenlicht}\label{thm:BBR}
Let $k$ be an algebraically closed field of characteristic zero.
Let $\phi: k^n \to k^n$ be a polynomial map.  If $\phi$ is injective,
then $\phi$ is surjective and the inverse is a polynomial map, i.e.,  
$\phi$ is a polynomial automorphism.
\end{theorem}

A useful tool to study the Jacobian conjecture is the Newton polygon. One source for this is \cite{CW2}, which we redefine here. Let $f=\displaystyle\sum_{i,j\geq 0} f_{ij}x^iy^j$ be a polynomial in $\mathcal{R}$ with $f_{ij}\in \mathbb{C}$. The \textit{support} of $f$ is defined as 
\[\supp(f)=\{(i,j)\in \mathbb{Z}^2 \mid f_{ij}\neq 0\} \subseteq \R^2.\]
The \textit{Newton polygon} for $f\in\mathcal{R}$, denoted $N(f)$,  is the convex hull of $\supp(f)$. 
The \textit{augmented Newton polygon} for $f\in\mathcal{R}$,  denoted $N^0(f)$,  is the convex hull of $\supp(f)\cup \{(0,0)\}$.\footnote{This is the definition of Newton polygon typically used in the 
Jacobian conjecture literature, e.g. \cite{CW2}.} 
A vertex of $N^0(f)$ is called a \emph{nontrivial vertex} if it is not equal to $(0,0)$. 
Clearly, $N(f)\subseteq N^0(f)\subseteq \mathbb{R}_{\ge0}^2$ for $f\in \mathcal{R}$.

Let $(F,G)$ be a Jacobian pair with $d=w\text{-deg}(F)$ and $e=w\text{-deg}(G)$ for some direction $w$.  
A generalized Magnus' formula (Theorem~\ref{Magnus_thm}), proved in \cite{HLLN}, says that $G$ is approximately equal to 
$$\sum_{\gamma\ge 0} c_\gamma F^{\frac{e-\gamma}{d}}$$
for some complex numbers $c_\gamma$. 
We develop an approach to the Jacobian conjecture aimed at minimizing the use of terms corresponding to $\gamma>0$.

As an initial step in this approach, we define and study the \emph{inner} (resp. \emph{innermost}) \emph{polynomials} of $F$ and $G$. The main result of this paper (Theorem~\ref{main_thm} and Corollary~\ref{main_cor}) shows that the northeastern vertex of the Newton polygon of each inner polynomial, called the \emph{inner} (resp. \emph{innermost}) \emph{vertex}, is located in a specific region. This result leads us to formulate a new conjecture (Conjecture~\ref{main_conj5}),  which implies the Jacobian conjecture. We prove a special case of this conjecture. In addition, as an important intermediate step towards the Jacobian conjecture,  we propose the inner vertex (resp. innermost vertex) conjecture (Conjecture~\ref{weak_Jac_conj}) and prove it for the case that $\frac{e}{d}>\frac{n-m}{a}-1$, where $a=d/\gcd(d,e)$, and $(m,n)$ is the northeastern vertex of the Newton polygon of $F$. Note that to understand this paper, the reader does not need to read the second and third papers \cite{GHLL2,GHLL3}.

Below is the outline of the paper.
\S2 introduces \Cref{Jac_conj} and demonstrates its equivalence to the Jacobian conjecture.
\S3 presents \Cref{Jac_conj2}, which imposes an extra condition on the Newton polygons of $F$ and $G$.
\S4 states that if $F$ can be reduced to a one-variable nonlinear polynomial, as defined within that section, the Jacobian conjecture will be proven. It also outlines a potential method to achieve this reduction.
In \S5, we propose Conjectures \ref{Jac_conj3} and \ref{Jac_conj4} based on the approach from \S4, which details the reduction of $F$ to a one-variable polynomial.
Then we present Theorem~\ref{main_thm} and Corollary~\ref{main_cor}, which provide a strong condition on the Newton polygons of the innermost  and inner polynomials of $F$. %
Theorem~\ref{main_thm} is proved in \S6.
In \S7, we propose Conjecture~\ref{main_conj5} based on the main result. 
In \S8, we prove a special case of this conjecture.
In Appendices A and B, we prove the uniqueness of the principal polynomial and an extended Magnus' formula. 

\noindent\emph{Acknowledgements.} 
We would like to thank Lenny Makar-Limanov and David Wright for valuable discussions, and Christian Valqui for numerous helpful suggestions. We also thank Rob Lazarsfeld for his advice. We extend our special thanks to Lauren K.~Williams, who generously helped us improve the readability of this work.

\section{\Cref{Jac_conj}}
In this section, we shall introduce \Cref{Jac_conj} and show that it implies
the Jacobian conjecture.
We start by recalling several fundamental results
on Jacobian pairs.
Let $\deg(f)$ denote the total degree of a polynomial $f$ (which is $(1,1)\text{-}\deg(f)$ as in Definition \ref{def:whomogeneous}).  

\begin{proposition}\label{prop:similar}
\cite[Theorem 10.2.1]{vdEssen}
If $(f,g)$ is a Jacobian pair with $\deg(f)>1$ and $\deg(g)>1$, then 
$N^0(f)$ is similar to $N^0(g)$ with the origin as center of similarity.
(It follows that the similarity ratio between the Newton polygons is
$\deg(f): \deg(g)$.)
\end{proposition}
%
%

The following result is a consequence of
\cite[5.1.6a and 5.1.11]{vdEssen}.
\begin{theorem}\label{thm:auto}
Let $f,g\in \C[x,y]$ 
	and suppose that 
the endomorphism  
	$\pi: \mathcal{R}\longrightarrow 
\mathcal{R}$\emph{ given by }
$\pi(x)=f$\emph{ and }$\pi(y)=g$ is an automorphism.
Then either 
	$\deg(f) \vert \deg(g)$ or 
$\deg(g) \vert \deg(f)$.
\end{theorem}

The following result is due to Abhyanker \cite[Theorem 10.2.23 (3) $\leftrightarrow$ (4)]{vdEssen}.
\begin{theorem}\label{thm:equiv}
The following statements are equivalent.
\begin{itemize}
\item For any Jacobian pair $(f,g)$, 
one has that either $\deg(f) \vert \deg(g)$ or $\deg(g) \vert \deg(f)$.
\item For any Jacobian pair $(f,g)$, the Newton polygons 
	$N^0(f)$ and $N^0(g)$ are triangles (where a line segment is considered
		to be a triangle).
\item The Jacobian conjecture.
\end{itemize}
\end{theorem}
We now present a conjecture equivalent to the Jacobian conjecture.


\begin{conjecturealpha}\label{Jac_conj}
Let $a,b\in \mathbb{Z}_{>0}$ be relatively prime with $2\le a < b$. 
Suppose that $F,G\in \mathcal{R}$ satisfy the following:

\noindent\emph{(1)} $[\ff,\GG]  \in  \mathbb{C}$;

\noindent\emph{(2)} $\{(0,0),(0,1),(1,0)\}\subseteq N^0(\ff)$, and $N^0(\ff)$ is similar to $N^0(\GG)$ with the origin as the center of similarity and with the ratio $\deg(\ff) : \deg(\GG) = a : b$. 

\noindent Then  $[\ff, \GG]=0$. 
\end{conjecturealpha}

Throughout the paper, we often use uppercase letters $\ff,\GG$ to denote
polynomials satisfying conditions (1) and/or (2) of 
Conjecture \ref{Jac_conj}, and  
lowercase letters $f,g$ to denote arbitrary polynomials.

\begin{proposition}
\Cref{Jac_conj} holds if and only if the Jacobian conjecture
holds.
\end{proposition}
\begin{proof}
For ``$\Rightarrow$'', assume that \Cref{Jac_conj} is true and the Jacobian conjecture is false. 
By Theorem \ref{thm:equiv}, 
there exists a Jacobian pair $(f,g)$ such that $\deg(f) \nmid \deg(g)$ and $\deg(g) \nmid \deg(f)$.
In particular, $\deg(f)>1$, $\deg(g)>1$. 
By \Cref{prop:similar}, 
$N^0(f)$ is similar to $N^0(g)$ with the origin as the center of similarity
and a similarity ratio of $\deg(f): \deg(g)$. 
If $\deg(f)/\deg(g)$ or $\deg(g)/\deg(f)$ is an integer,
then the first statement of \Cref{thm:equiv} holds.
Otherwise, without loss of generality we assume $\deg(f)/ \deg(g) = a/b$
for relatively prime positive integers $a,b$ such that 
$2 \leq a  < b$. Since the degrees of both $f$ and $g$ are 
greater than $1$, \cite[Proposition 10.2.6]{vdEssen}
implies that the Newton polygons of both $f$ and $g$ contain
 $\{(0,0), (0,1), (1,0)\}$.  So the first two conditions of 
\Cref{Jac_conj} hold, and hence $[f,g]=0$, a contradiction.
Thus, we have shown that if $(f,g)$ is a Jacobian pair,
the first statement of \Cref{thm:equiv} holds.  Therefore, 
by \Cref{thm:equiv}, the Jacobian conjecture holds.

For ``$\Leftarrow$'', suppose that the Jacobian conjecture holds.
Suppose that $F,G$ satisfy conditions (1), (2) of \Cref{Jac_conj}. 
If $[F,G]\in\mathbb{C}\setminus\{0\}$, then 
by \Cref{thm:equiv}, either $\deg(F)|\deg(G)$ or $\deg(G)|\deg(F)$.
This contradicts condition (2) of \Cref{Jac_conj}. So $[F,G]=0$.
\end{proof}

\section{Conjecture~\ref{Jac_conj2}}
In this section, we introduce 
\Cref{Jac_conj2}, which is motivated
by known results about the Newton polygon of a potential
counterexample to the Jacobian conjecture.  
We then 
 show that \Cref{Jac_conj2} implies \Cref{Jac_conj}.

\begin{definition}
For any nonnegative integers $m,n\in \mathbb{Z}$, we define
$$
\Rect_{m,n} =\textrm{ the convex hull of }\{(0,0),(m,0),(m,n),(0,n)\}.
$$
See Figure \ref{fig:NNN}.

\end{definition}

\begin{figure}[h]
\begin{center}
\begin{tikzpicture}[scale=0.3]
\begin{scope}[shift={(40,0)}]
\draw (0,8)--(4,8)--(4,0);
\fill[black!10] (0,0)--(0,8)--(4,8)--(4,0)--(0,0);
\draw (4, 8) node[anchor=west] {\tiny $(m,n)$};
\draw (4, 0) node[anchor=north] {\tiny $(m,0)$};
\draw (0, 8) node[anchor=east] {\tiny $(0,n)$};
\draw (2, 1) node[anchor=south] {\tiny $\Rect_{m,n}$};
\draw[->] (0,0) -- (6,0)
node[above] {\tiny $x$};
\draw[->] (0,0) -- (0,9)
node[right] {\tiny $y$};
\end{scope}
\end{tikzpicture}
\end{center}
	\caption{$\Rect_{m,n}$}
\label{fig:NNN}
\end{figure}

The following result is well-known; see \cite{Na1, ML1} and references therein.
\begin{theorem}\label{thm:trap}
Suppose that $f,g \in \C[x,y]$ form a Jacobian
pair. If the polynomial map $\RRR\to\RRR$
sending $(x,y)\mapsto (f,g)$ fails to be 
bijective, then there exists an automorphism 
$\xi \in \Aut(\RRR)$ such that 
the Newton polygon $N^0(\xi(f))$ 
contains a vertex $(m,n)$
	and $N^0(\xi(f))$ is 
	contained in the rectangle $\Rect_{m,n}$.
\end{theorem}

We have the following convenient variation of the above result.

\begin{corollary}\label{cor:trap}
Suppose that $f,g \in \C[x,y]$ form a Jacobian
pair. If the polynomial map $\RRR\to\RRR$
sending $(x,y)\mapsto (f,g)$ fails to be 
bijective, then there is an automorphism 
$\zeta \in \Aut(\RRR)$ such that 
the Newton polygon $N^0(\zeta(f))$ is precisely the 
rectangle $\Rect_{m,n}$ for some $m,n$.
\end{corollary}
\begin{proof}
Since the polynomial $\xi(f)$ in \Cref{thm:trap} contains
$(m,n)$ as the northeast-most vertex of its Newton polygon,
applying 
the automorphism $x\mapsto x+c_x, y\mapsto y+c_y$ (for generic $c_x,c_y\in\mathbb{C}$)  
to $\xi(f)$ (and in particular the monomial $x^m y^n$)
results in the Newton polygon $\Rect_{m,n}$.
\end{proof}

\begin{definition}\label{def:whomogeneous}
A nonzero element $(u,v)\in \Z^2$ is called a \emph{direction} if $\gcd(u,v)=1$ and $u>0$ or $v>0$.  Let $\mathcal{D}$ be the set of all directions. 
	To each such  direction we consider its $(u,v)$-grading on $\mathcal{R}$.
	So $\mathcal{R} = \oplus_{n\in \mathbb{Z}} \mathcal{R}^{(u,v)}_n$, where $\mathcal{R}^{(u,v)}_n$ (sometimes denoted by $\mathcal{R}_n$ if $(u,v)$ is clear from the context) is the $\mathbb{C}$-vector space generated by the monomials $x^iy^j$ with
$ui + vj = n$. 
A non-zero element $f$ of $\mathcal{R}_n^{(u,v)}$ is called a 
	\emph{$(u,v)$-homogeneous} element of $\mathcal{R}$, 
	and $n$ is called its \emph{$(u,v)$-degree}, denoted $(u,v)\text{-}\deg(f)$. 

We can write  any nonzero $f\in \mathcal{R}$ as a sum of 
	$(u,v)$-homogeneous polynomials $f=\sum_n f^{(u,v)}_n$,
	where $f^{(u,v)}_n\in \mathcal{R}^{(u,v)}_n$.  
	The element of highest $(u,v)$-degree in the homogeneous decomposition of
	$f$ is called its \emph{$(u,v)$-leading form} and is denoted by $f_+$. The
	$(u,v)$-degree of $f$ is by definition $(u,v)\text{-}\deg(f_+)$. 

We denote 
$\deg(f)=(1,1)\text{-}\deg(f)$, $\deg_x(f)=(1,0)\text{-}\deg(f)$, and $\deg_y(f)=(0,1)\text{-}\deg(f)$.

For a single or multivariable polynomial (or power series) $f=\sum_{\alpha}c_{\alpha}\mathbf{x}^\alpha$, denote the coefficient $[\mathbf{x}^\alpha]f=[f]_{\mathbf{x}^\alpha}=c_{\alpha}$. 

\end{definition}

\begin{definition}
	For each pair $(m,n)\in \mathbb{Z}_{>0}^2$, 
	
	let $R_{m,n}$ be the set of polynomials $f\in \mathcal{R}$ with $N^0(f)=\Rect_{m,n}$ and $[x^my^n]f=1$; 
	
	let $\overline{R}_{m,n}$ be the set of polynomials $f\in \mathcal{R}$ with $N^0(f)\subseteq\Rect_{m,n}$ and $[x^my^n]f\neq 0$.
 \end{definition} 

\begin{definition}
We set 
$$\mathcal{Q}=\{(a,b,m,n)\in \mathbb{Z}_{>0}^4\, : \,  \ a|m,\ a|n, \  \gcd(a,b)=1\text{ and }2\le a < b\}.$$
\end{definition}

Throughout this paper, we fix $(a,b,m,n)\in \mathcal{Q}$.

We now state a new
conjecture which we will subsequently show 
implies \Cref{Jac_conj} (see \Cref{BtoA}), 
and hence implies the Jacobian conjecture.

\begin{conjecturealpha}\label{Jac_conj2}
Suppose that $\ff,\GG\in \mathbb{C}[x,y]$ satisfy the following:

\noindent\emph{(1)} $[\ff,\GG]  \in  \mathbb{C}$;

\noindent\emph{(2)} $\ff\in R_{m,n}$ and $\GG\in R_{bm/a,bn/a}$.
(In particular, $\{(0,0),(0,1),(1,0)\}\subseteq N^0(\ff)$, and $N^0(\ff)$ is similar to $N^0(\GG)$ with
	the origin as center of similarity and with the ratio 
	$a : b$.)

\noindent Then  $[\ff, \GG]=0$. 
\end{conjecturealpha}

\begin{proposition}\label{BtoA}
Conjecture~\ref{Jac_conj2} implies Conjecture~\ref{Jac_conj}. 
\end{proposition}

\begin{proof}
Assume \Cref{Jac_conj2} is true. We want to prove \Cref{Jac_conj}.
So consider $f,g\in \mathbb{C}[x,y]$ such that properties (1)
and (2) of \Cref{Jac_conj} hold.  That is, we have 
$[f,g]\in \mathbb{C}$, $\{(0,0),(0,1),(1,0)\} \subseteq N^0(f)$,
and $N^0(f)$ and $N^0(g)$ are similar with the origin
as center of similarity and with similarity ratio $\deg(f):\deg(g) = a:b$,
with $a$ and $b$ relatively prime with $2 \leq a<b$.

Suppose that 
 $[f,g]\in \C\setminus \{0\}$.  
	If the homomorphism $\pi: \mathcal{R} \to \mathcal{R}$ given by 
	$\pi(x)=f$ and $\pi(y)=g$ is an automorphism, 
	then by \Cref{thm:auto}, 
	either $\deg(f) \vert \deg(g)$ or $\deg(g) \vert \deg(f)$,
	contradicting our hypothesis that $\deg(f):\deg(g)=a:b$.
Therefore $\pi$ fails to be an automorphism, and 
	hence by \Cref{thm:BBR} it fails to be bijective. 
	By 
\Cref{cor:trap},
we can apply an automorphism if needed so as 
to assume that 
$N^0(f) = \Rect_{m,n}$ for some $m,n$.  Therefore by our hypothesis, 
	$N^0(g) = \Rect_{bm/a, bn/a}$.

By multiplying each of $f$ and $g$ by some nonzero constants if necessary, we can assume that 
	$[x^m y^n]f=1$ and $[x^{bm/a} y^{bn/a}]g=1$.  Therefore $f\in R_{m,n}$ and  $g\in R_{bm/a, bn/a}$.  By \Cref{Jac_conj2},  $[f,g]=0$.
	This shows that  Conjecture~\ref{Jac_conj2} implies Conjecture~\ref{Jac_conj}.
\end{proof}

\section{A step towards reducing $F$ to a one-variable polynomial}
A possible proof of Conjecture~\ref{Jac_conj2} is to show that 
if $F$ satisfies the hypotheses of \Cref{Jac_conj2}, then 
there is a nontrivial way to write $F$ as a specialization of a one-variable nonlinear 
polynomial,
that is, $F=\alpha(W^\circ)$ for some polynomial $W^\circ\in \mathcal{R}$ and 
one-variable polynomial $\alpha(z)\in \mathbb{C}[z]$ with $\deg(\alpha)\ge2$.  
 \Cref{FG00} shows that in this case,
if $[F,G]\in \C$, then in fact $[F,G]=0$.

In this section we will construct a polynomial $W_F$ that we call 
the \emph{$F$-generator}, which is the most natural candidate for $W^\circ$. In particular, we will prove \Cref{existence_of_Ec}, which shows that 
this candidate exists.  We will eventually prove that there exists a one-variable
nonlinear polynomial $\alpha_F$ such that 
$F=\alpha_F(W_F)$.


\subsection{Principal polynomials}

There are in general 
many ways to write a polynomial 
$f\in \mathcal{R}$  
as $\alpha(W)$ for $W\in \mathcal{R}$ and $\alpha(z)\in \mathbb{C}[z]$.
For example, for any 
$f\in \mathcal{R}$  
and $c\in \C$, if we set $W_c=f-c$ and $\alpha_c=z+c$, 
then we have $f=\alpha_c(W_c)$.
However, if we require that 
the subleading coefficient of $\alpha(z)$ vanishes, we significantly reduce the number of choices.

\begin{definition}\label{T}
Let $\mathbb{T}\subseteq \mathbb{C}[z]$ be the set of  \emph{Tschirnhausen polynomials}, that is, 
$$\mathbb{T}=\{ \g(z)=z^k + e_{k-1} z^{k-1} + \cdots + e_0 z^0 \in \mathbb{C}[z] \, : \, k\in \mathbb{Z}_{>0}, \, e_{k-1}=0,\text{ and }e_{k-2},\cdots,e_0\in \mathbb{C} \}.$$ 
\end{definition}

\begin{lemma}
If we can write 
$f=\alpha(W)$ 
where $\alpha(z)\in \C[z]$ and $W\in \RRR$,
then we can write 
$f=\alpha'(W')$ where $\alpha'(z)\in \T$ has the same degree as $\alpha(z)$
and $W'\in \RRR$.
\end{lemma}
\begin{proof}
Suppose that $f=\alpha(W)$ for some polynomial $\alpha(z)$ of degree $k$. 
By multiplying $W$ by a constant if necessary,
we can assume $[z^k] \alpha(z)=1$.
Factor $\alpha(z)=(z-r_1)\dots (z-r_k)$ and let 
	$s:=\frac{1}{k} \sum_{i=1}^k r_i$.
	Set $\alpha'(z)=(z-(r_1-s))\dots (z-(r_k-s))$, and $W':=W-s$.
	Then $\alpha'(z)$ lies in $\T$, and 
	$\alpha'(W')=\alpha(W)=f$.
\end{proof}

\begin{definition}
Given
	a polynomial $f\in \mathcal{R}$, we let 
	$$\mathcal{W}(f)=\{ W\in \mathcal{R} \, : \, f=\g(W)\text{ for some }\g(z)\in \mathbb{T}\}.$$
	Since we can always write   $f=\alpha(W)$ for $\alpha(z)=z\in \mathbb{T}$ and $W=f$,
	we have that  $f\in \mathcal{W}(f)$.  
	 If 
	$f\in \mathcal{R}$ has the property $\mathcal{W}(f) = \{f\}$, we
	call $f$
	a \emph{principal} polynomial.\footnote{In \cite{N0}, Nowicki refers
	to such polynomials as \emph{closed}. 
	More specifically, he calls a polynomial $f$ \emph{closed} if
	$\{g\in k[x,y]  \ \vert \ [f,g]=0\} = k[f]$.
	}
\end{definition}

\begin{lemma}\label{lem:unique}
Let $f\in \RRR$.  If $f=\alpha(W)$ where $\alpha(z)\in \C[z]$ and $W\in \RRR\setminus \C$,
and also $f=\beta(W)$ where $\beta(z)\in \C[z]$, then $\alpha(z)=\beta(z)$.
\end{lemma}
\begin{proof}
Let $\gamma(z)=\alpha(z)-\beta(z)$.  
Since $\alpha(W)=\beta(W)$, we have $\gamma(W)$=0.  Since $W$ is not a constant,
this implies that every coefficient of $\gamma$ must vanish, so $\alpha=\beta$.
\end{proof}

\begin{theorem}\cite[Theorem 1.4]{N0}\label{thm:N0}
	If $f,g$ are nonconstant polynomials in $\mathcal{R}$ such that $[f,g]=0$, then
	there exists $h\in \mathcal{R}$ and $\alpha(z), \beta(z)\in \mathbb{C}[z]$
	such that $f=\alpha(h)$ and $g=\beta(h)$.
\end{theorem}

The following corollary is immediate from \Cref{thm:N0}.
\begin{corollary}
Suppose that $f\in \mathcal{R}$ is a principal polynomial.  Then 
if $g\in \mathcal{R}$ is a nonconstant polynomial satisfying $[f, g]=0$, 
there exists $\beta(z)\in \mathbb{C}[z]$ such that $g=\beta(f)$. 
\end{corollary}

\begin{example}
Let $f(x,y)=(x^2y+x^2+y)^{12}+(x^2y+x^2+y)^4+1$. Here are some examples of pairs 
	$(W(x,y),\alpha(z))$ such that $f(x,y)=\alpha(W(x,y))$:

	$W=f$, $\g(z)=z$;

	$W=\omega(x^2y+x^2+y)^4$, $\g(z)=z^3+\omega^{-1}z+1$, where $\omega^3=1$;

	$W=\omega (x^2y+x^2+y)^2$, $\g(z)=z^6+\omega^{-2}z^2+1$,  where $\omega^6=1$;

	$W=\omega (x^2y+x^2+y)$, $\g(z)=z^{12}+\omega^{-4}z^4+1$,  where $\omega^{12}=1$.

Note that the polynomial $(x^2y+x^2+y) \in 
	\mathcal{W}(f)$ 
	is itself a principal polynomial.
\end{example}

\begin{lemma}\label{FG00}
Assume that $[F,G]  \in  \mathbb{C}$. Suppose that $F$ is not principal, that is, $F=\alpha(W^\circ)$ for some $W^\circ\in \mathcal{R}$ and some $\alpha(z)\in \mathbb{C}[z]$ with $\deg(\alpha)\ge2$. Then $[\ff,\GG]=0$.
\end{lemma}
\begin{proof}
If $W^\circ\in \mathbb{C}$ then $F=\alpha(W^\circ)\in \mathbb{C}$, so $[\ff,\GG]=0$.

Suppose $W^\circ\in \mathcal{R}\setminus \mathbb{C}$. Since $F=\alpha(W^\circ)$, we get
$$
[F, G]  = \frac{d\g}{dz}(W^\circ) [W^\circ, G].
$$
Since $\deg\g>1$, we
	have $\frac{d\g}{dz}(W^\circ)
	\in \mathcal{R}\setminus \mathbb{C}$.
	Since $[F,G]\in \mathbb{C}$,
	we must have $[W^\circ,G]=0$, thus $[\ff, G]=0$. 
\end{proof}

\begin{remark}
Lemma~\ref{unique_principal} in Appendix A shows that $\mathcal{W}(f)$ contains a unique principal polynomial up to roots of unity. 
\end{remark}

\subsection{Construction of the $F$-generator $W_F$}

Let $(a, b, m, n) \in \mathcal{Q}$. 
For $\ff\in R_{m,n}$, denote $A=(m,n)\in \mathbb{R}^2$, 
$\mathcal{N}'=\Rect_{{m}/{a},n/a}$,
 and $\mathcal{N}''=\mathcal{N}'+\frac{a-1}{a}\overrightarrow{OA}$.
See Figure \ref{fig:example}.

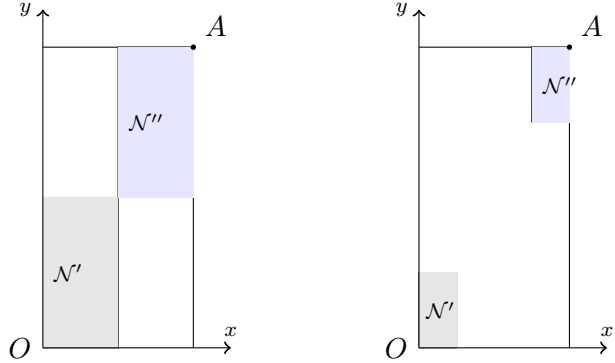
\begin{figure}[h]
\begin{center}
\begin{tikzpicture}[scale=0.50]
\usetikzlibrary{patterns}
\draw (0,8)--(4,8)--(4,0);
\draw (0,0)--(0,4)--(2,4)--(2,0)--(0,0);
\fill[black!10] (0,0)--(0,4)--(2,4)--(2,0)--(0,0);
\draw (0,2) node[anchor=west] {\tiny $\mathcal{N}'$};
\draw (2,4)--(2,8)--(4,8)--(4,4)--(2,4);
\fill[blue!10] (2,4)--(2,8)--(4,8)--(4,4)--(2,4);
\draw (2,6) node[anchor=west] {\tiny $\mathcal{N}''$};
\draw (4,8) node {\huge .};
\draw (4, 8) node[anchor=south west] {\small $A$};
\draw (0, 0) node[anchor=east] {\small $O$};
\draw[->] (0,0) -- (5,0)
node[above] {\tiny $x$};
\draw[->] (0,0) -- (0,9)
node[left] {\tiny $y$};

\begin{scope}[shift={(10,0)}]
\usetikzlibrary{patterns}
\draw (0,8)--(4,8)--(4,0);
\draw (0,0)--(0,2)--(1,2)--(1,0)--(0,0);
\fill[black!10] (0,0)--(0,2)--(1,2)--(1,0)--(0,0);
\draw (-.1,1) node[anchor=west] {\tiny $\mathcal{N}'$};
\draw (3,6)--(3,8)--(4,8)--(4,6)--(3,6);
\fill[blue!10] (3,6)--(3,8)--(4,8)--(4,6)--(3,6);
\draw (3,7) node[anchor=west] {\tiny $\mathcal{N}''$};
\draw (4,8) node {\huge .};
\draw (4, 8) node[anchor=south west] {\small $A$};
\draw (0, 0) node[anchor=east] {\small $O$};
\draw[->] (0,0) -- (5,0)
node[above] {\tiny $x$};
\draw[->] (0,0) -- (0,9)
node[left] {\tiny $y$};
\end{scope}
\end{tikzpicture}
\end{center}
\caption{$N^0(F)$, $\mathcal{N}'=N^0(Q)$, and $\mathcal{N}''$ for  $a=2$ (left) and $a=4$ (right).}
\label{fig:example}
\end{figure}

\begin{theorem}\label{existence_of_Ec}
Let $(a,b,m,n)\in\mathcal{Q}$.
For each $F\in R_{m,n}$, there exists a unique pair $(\beta, W_F)\in \T \times  \mathcal{R}$ satisfying the following:

\noindent $\bullet$ $W_F$ is a principal polynomial and lies in 
	$R_{m/(a\delta), n/(a\delta)}$,
where 
 $\delta$  divides both
	 $\frac{m}{a}$ 
	and $\frac{n}{a}$.

\noindent $\bullet$ $\beta(z)$ has degree $\delta$ and ${\rm supp}(\ff-(\beta(W_F))^a)\subseteq N^0(\ff)\setminus \mathcal{N}''$.

	The polynomial $W_F$ will be called the \emph{$F$-{generator}} 
	or the \emph{generator for }$F$.
\end{theorem}

\begin{remark}
$\beta(z)^a$ is our first approximation to 
	the polynomial $\alpha_F(z)$ such that $\alpha_F(W_F)=F$.
\end{remark}


To prove Theorem~\ref{existence_of_Ec}, we need \Cref{lem:c=p^2 most generalized} below, whose proof is similar to \cite[Lemma 2.6]{HLLN}.  When we prove \Cref{existence_of_Ec} using 
\Cref{lem:c=p^2 most generalized}, $Q$ will play the role of $\beta(W_F)$.

\begin{lemma}
	\label{lem:c=p^2 most generalized}
	Let $(a,b,m,n)\in\mathcal{Q}$ and 
let $\ff\in R_{m,n}$.   Then there exists a unique polynomial 
	$Q\in R_{{m}/{a}, {n}/{a}}$
	such that 
	${\rm supp}(\ff-Q^a)\subseteq N^0(\ff)\setminus \mathcal{N}''$.

Moreover, there exists 
	a principal polynomial $W^\circ\in\mathcal{W}(Q)$ (unique up to root of unity),
	and a polynomial
	 $\beta(z)\in\mathbb{T}$ 
	of degree $\delta\in\mathbb{Z}_{>0}$ such that 
          $Q=\beta(W^\circ)$, and 
	$N^0(W^\circ)=\frac{1}{\delta a} N^0(\ff) \;(=\frac{1}{\delta a} \Rect_{m,n} )$. 
\end{lemma}

\begin{proof}[Proof of Lemma~\ref{lem:c=p^2 most generalized}]
Fix two positive numbers $r_1\in\mathbb{Q}$ and $r_2\in \mathbb{R}\setminus\mathbb{Q}$. 
Note that $\{(x,y)\ | r_1 x+r_2 y=r_1 m+r_2 n\}\cap \mathbb{Z}^2=\{(m,n)\}$.  Then $N^0(F)$ lies in the half plane $r_1 x+r_2 y\le r_1 m+ r_2 n$. Let $n'$ be the number of lattice points in $\mathcal{N}'$. 
We use $(r_1,r_2)$, viewed as a linear functional, to put a total order 
$z_1> \dots > z_n$ on the lattice points in $\mathcal{N}'$: specifically, we will 
label the lattice points of $\mathcal{N}'$ by 
$\{z_i=(x_i,y_i)\}_{1\le i\le n'}$, where $r_1 x_1+r_2 y_1>r_1 x_2+r_2 y_2>\cdots>r_1 x_{n'}+r_2 y_{n'}$. Then $z_1=(m/a,n/a)$. For $i=1,\dots,n'$, denote ${\bf x}^{z_i}=x^{x_i}y^{y_i}$. 

Our goal is to construct a polynomial $Q=\sum_{i=1}^{n'} q_i {\bf x}^{z_i}$ with 
the desired properties.  We will construct the coefficients
 $q_1,..,q_{n'}$ inductively.  We set  $q_{1}=1$, so that 
	$Q\in R_{\frac{m}{a}, \frac{n}{a}}$.
 For the  inductive step, 
assume that $q_{1},\dots,q_{k-1}$ (for some $k>1$) 
have already been  determined. 

Note that the lattice points of $\mathcal{N}''$ 
are precisely the points 
$(a-1)z_1+z_i$ for $i=1\dots,n'$.
Since we want ${\rm supp}(\ff-Q^a)\subseteq N^0(\ff)\setminus \mathcal{N}''$,
we need to make sure that for each $k$, the coefficients
$[{\bf x}^{(a-1)z_1+z_k}]F$ and 
$[{\bf x}^{(a-1)z_1+z_k}](\sum_{i=1}^{n'} q_{i}{\bf x}^{z_i})^a$ are equal.
Note that because our ordering 
$z_1 > \dots > z_{n'}$ is induced by  a linear functional, the 
coefficient $[{\bf x}^{(a-1)z_1+z_k}](\sum_{i=1}^{n'} q_{i}{\bf x}^{z_i})^a$ is equal to 
the coefficient
$[{\bf x}^{(a-1)z_1+z_k}](\sum_{i=1}^k q_{i}{\bf x}^{z_i})^a$.
(Any expression for $(a-1)z_1+z_k$ as a linear combination of 
	$a$ of the $z_i$'s will necessarily use the variables $z_i$ for $i\leq k$.)

We now claim that the coefficient $q_k$ is uniquely
	determined by the condition that the coefficients
$[{\bf x}^{(a-1)z_1+z_k}]F$  and
	$[{\bf x}^{(a-1)z_1+z_k}](\sum_{i=1}^{k} q_{i}{\bf x}^{z_i})^a$ are equal.
	To see this, note that 
	the coefficient of ${\bf x}^{(a-1)z_1+z_k}$
	in $(\sum_{i=1}^{k} q_{i}{\bf x}^{z_i})^a$ equals 
$${a\choose 1}q_1^{a-1}q_{k}+\sum_{\stackrel{z_{i_1}+\cdots+z_{i_a}=(a-1)z_1+z_k}{i_1,\dots,i_a<k}}q_{i_1}\cdots q_{i_a},$$ so requiring that this coefficient equal a fixed number will
 uniquely determine  $q_{k}$. 

For the second statement of the 
lemma, let $W^\circ\in\mathcal{W}(Q)$ be a principal polynomial; 
it is unique up to a root of unity 
(by 
\Cref{unique_principal}).
	Then by \Cref{lem:unique}, there is a unique 
 $\beta(z)\in\mathbb{T}$, say of degree $\delta$,  
 such that $Q=\beta(W^\circ)=(W^\circ)^\delta + e_{\delta-2} (W^\circ)^{\delta-2} + 
 \cdots + e_0$.

Note that for any $f\in \RRR$, the Newton polygon $N^0(f^i)$ is just the $i$th dilation 
$iN^0(f)$ of the Newton polygon $N^0(f)$. Therefore 
$N^0(\beta(W^0)) = \delta N^0(W^\circ)$.
  Therefore  we conclude that $N^0(Q)=\delta N^0(W^\circ)$, hence 
	$N^0(W^\circ)=\frac{1}{\delta a} N^0(\ff)$.
\end{proof}


\begin{proof}[Proof of Theorem~\ref{existence_of_Ec}]
By \Cref{lem:c=p^2 most generalized}, there exists a unique polynomial $Q$ such that 
$N^0(Q - x^{m/a} y^{n/a})\subsetneq N^0(Q)=\mathcal{N}'$ 
and ${\rm supp}(\ff-Q^a)\subseteq N^0(\ff)\setminus \mathcal{N}''$.
We now want to use 
	\Cref{lem:c=p^2 most generalized}
	to produce 
	a principal polynomial $W^{\circ}\in\mathcal{W}(Q)$ and 
$\beta(z)\in\mathbb{T}$ (of degree $\delta$) such that $Q=\beta(W^\circ)$.
Since $[x^{m/a} y^{n/a}]Q=1$, the coefficient $[x^{m/a\delta}y^{n/a\delta}]W^{\circ}$ must be a root of unity.
Choose $W_F$ to be the unique polynomial among the principal polynomials in $\mathcal{W}(Q)$ 
such that $[x^{m/a\delta}y^{n/a\delta}]W_F=1$; this determines $\beta$ such that 
	$Q=\beta(W_F)$.
By  
	\Cref{lem:c=p^2 most generalized},
	we have 
$N^0(W_F)= 
\frac{1}{\delta a} \Rect_{m,n}$. 

	Since $Q=\beta(W_F)$, we have ${\rm supp}(\ff-\beta(W_F)^a)\subseteq N^0(\ff)\setminus \mathcal{N}''$, which 
completes our proof of the theorem.
\end{proof}

\begin{definition}\label{df:pre-generator}
The polynomial $Q$ constructed in Lemma~\ref{lem:c=p^2 most generalized} is called the \emph{pre-generator} for the pair $(F,a)$. The polynomial $\zeta(F):=F-Q^a$ is called the \emph{inner polynomial} of the pair $(F,a)$. When $a$ is clear from context, it will be called  the inner polynomial of $F$.  
It follows from Lemma \ref{lem:c=p^2 most generalized} that 
	${\rm supp}\ \zeta(F)\subseteq N^0(\ff)\setminus \mathcal{N}''$.
\end{definition}

\section{Conjecture~\ref{Jac_conj3}, Conjecture~\ref{Jac_conj4}, and the main theorem}

In this section, we introduce two new conjectures (Conjectures~\ref{Jac_conj3} and~\ref{Jac_conj4}) which imply Conjecture~\ref{Jac_conj2} and hence the Jacobian conjecture. Then we state our main theorem.

\begin{conjecturealpha}\label{Jac_conj3}
Let $(a, b, m, n) \in \mathcal{Q}$. 
Suppose that $\ff,\GG\in \mathbb{C}[x,y]$ satisfy the following:

\noindent\emph{(1)} $[\ff,\GG]  \in  \mathbb{C}$;

\noindent\emph{(2)} $\ff\in R_{m,n}$ and $\GG\in R_{bm/a,bn/a}$.

Let $\EEE$ be the F-generator constructed in 
\Cref{existence_of_Ec}.
Then  there exists $\g\in \mathbb{T}$ with $\deg \g \geq 2$ such that 
	$\ff = \g(\EEE)$.
 \end{conjecturealpha}

\begin{remark}\label{rem:GGV}
One may compare Conjecture~\ref{Jac_conj3} with a conjecture 
	from \cite{GGV14}, which has a similar flavor of expressing $F$
	in terms of a one-variable polynomial.
\end{remark}

For integers $r_1\leq r_2$, we use the notation
$[r_1,r_2]:=\{x\in \mathbb{Z} \ : \ r_1\le x\le r_2\}$; 
this should not be confused with the Jacobian $[f,g]$.

The following definition provides a way to construct the 
polynomial $\alpha$ as in \Cref{Jac_conj3}. 

\begin{definition}
Let $a,m,n\in \mathbb{Z}_{>1}$ such that $a|m$, $a|n$, 
and $\ff\in R_{m,n}$.
Let $\EEE$ be the $\ff$-generator constructed in 
\Cref{existence_of_Ec}, and 
 $Q$ be the pre-generator (see Definition \ref{df:pre-generator})
	so that 
	$N^0(\EEE) \subseteq \Rect_{m/{a\delta}, n/{a\delta}}$,
	and ${\rm supp}(\ff-Q^a)\subseteq \Rect_{m,n} \setminus \mathcal{N}''$.

Let $e_0,e_1, e_2, ...,e_{(a-1)\delta-1}\in\mathbb{C}$ 
	be uniquely determined by the condition that the polynomial
	\begin{equation}\label{eq:Z}
		Z:=F-Q^a-\sum_{j=0}^{(a-1)\delta-1} e_{j}\EEE^{j}
	\end{equation} does not contain in its support any
	term of the form $x^{jm/(a\delta)}y^{jn/(a\delta)}$ for $j\in[0,(a-1)\delta-1]$.
	
	This polynomial $Z$ is called \emph{the innermost polynomial} of $F$. 
\end{definition}

\begin{conjecturealpha}\label{Jac_conj4}
Let $(a, b, m, n) \in \mathcal{Q}$. 
Suppose that $\ff,\GG\in \mathbb{C}[x,y]$ satisfy the following:

\noindent\emph{(1)} $[\ff,\GG]  \in  \mathbb{C}$;

\noindent\emph{(2)} $\ff\in R_{m,n}$ and $\GG\in R_{bm/a,bn/a}$.

Then $Z=0$. 
 \end{conjecturealpha}
We will show that \Cref{Jac_conj4} implies Conjecture~\ref{Jac_conj3} (see Lemma~\ref{DtoC}).

See \Cref{fig:ConjD} for a depiction of the lattice points relevant to 
\eqref{eq:Z}.

\begin{figure}[h]
\begin{center}
\begin{tikzpicture}[scale=0.50]
\usetikzlibrary{patterns}

\begin{scope}[shift={(10,0)}]
\usetikzlibrary{patterns}
\draw (0,8)--(4,8)--(4,0);
\draw (3,6)--(3,8)--(4,8)--(4,6)--(3,6);
\fill[blue!10] (3,6)--(3,8)--(4,8)--(4,6)--(3,6);
\draw (3,7) node[anchor=west] {\tiny $\mathcal{N}''$};
\draw (4,8) node {\huge .};
\draw (2,4) node {\tiny $\bullet$};
\draw (1,2) node {\tiny $\bullet$};
\draw (0,0) node {\tiny $\bullet$};
\draw (2.5,5) node {\tiny $\bullet$};
\draw (1.5,3) node {\tiny $\bullet$};
\draw (0.5,1) node {\tiny $\bullet$};
	\draw (4, 8) node[anchor=south west] {\small $(m,n)$};
\draw[->] (0,0) -- (5,0)
node[above] {\tiny $x$};
\draw[->] (0,0) -- (0,9)
node[left] {\tiny $y$};
\end{scope}
\end{tikzpicture}
\end{center}
	\caption{$N^0(F)=\Rect_{m,n}$ and $\mathcal{N}''$, shown for the case $m=8$,
	$n=16$, and $a=4$.  The lattice points in bold are those of the form
	 $(\frac{jm}{a\delta}, \frac{jn}{a\delta})$ for $j\in[0,(a-1)\delta-1]$, where 
	 $\delta=2$.}
\label{fig:ConjD}
\end{figure}
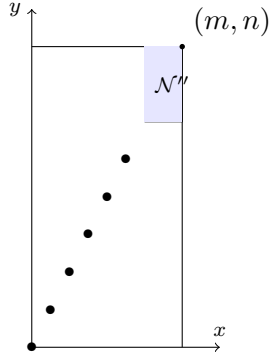

\begin{lemma}
The coefficients $e_j$ 
from \Cref{Jac_conj4}
	are well-defined.
\end{lemma}
\begin{proof}
\Cref{Jac_conj4} concerns the 
lattice points of the form 
	$(j \cdot \frac{m}{a\delta}, j \cdot \frac{n}{a\delta})$, for $j\in \Z_{\geq 0}$,
in $\Rect_{m,n}$.  Since 
 ${\rm supp}(\ff-Q^a)\subseteq \Rect_{m,n} \setminus \mathcal{N}''$, 
the only lattice points of this form which may have nonempty
support in $\ff-Q^a$ are those of the form 
$(j \cdot \frac{m}{a\delta}, j \cdot \frac{n}{a\delta})$
where $0 \leq j \leq (a-1)\delta -1$.

Since 
$W_F$ lies in $R_{m/(a\delta), n/(a\delta)}$,
and hence $(W_F)^j$ lies in $R_{jm/(a\delta), nj/(a\delta)}$,
we can inductively construct 
$e_{(a-1)\delta-1}, \dots, e_2, e_1, e_0 \in\mathbb{C}$ so that 
	\eqref{eq:Z} holds.
		\end{proof}        

Next we show that Conjecture~\ref{Jac_conj4} implies Conjecture~\ref{Jac_conj3}, and Conjecture~\ref{Jac_conj3} implies Conjecture~\ref{Jac_conj2}, hence the Jacobian conjecture.

\begin{lemma}\label{CtoB}
Conjecture~\ref{Jac_conj3} implies Conjecture~\ref{Jac_conj2}.
\end{lemma}
\begin{proof}
Assuming Conjecture~\ref{Jac_conj3} is true, it suffices to prove the following: if there is $\g\in \mathbb{T}$ such that $\ff - \g(\EEE)=0$, then $[F,G]=0$. But this is proved in Lemma~\ref{FG00}.
\end{proof}

\begin{lemma}\label{DtoC}
Conjecture~\ref{Jac_conj4} implies Conjecture~\ref{Jac_conj3}.
\end{lemma}
\begin{proof}
Suppose that the hypotheses of \Cref{Jac_conj3} hold, and   
suppose we know that \Cref{Jac_conj4} holds.  
Let $\EEE$ be the $F$-generator constructed in 
\Cref{existence_of_Ec}.  

	Let $\alpha_F(z):=\beta(z)^a +
		\sum_{j=0}^{(a-1)\delta-1} e_{j} z^{j},$
		where the coefficients $e_j$ are those 
		constructed in \Cref{Jac_conj4}.
		We know that $Q = \beta(W_F)$ (by the 
		proof of 
	\Cref{existence_of_Ec}), so 
	$\alpha_F(W_F) 
	= (\beta(W_F))^a 
+	\sum_{j=0}^{(a-1)\delta-1} e_{j} (W_F)^{j} = 
	 Q^a  + 
		\sum_{j=0}^{(a-1)\delta-1} e_{j} (W_F)^{j}.$  By 
	\Cref{Jac_conj4}, this quantity equals $F$.
	Therefore $F = \alpha_F(W_F)$, which is the 
	conclusion of \Cref{Jac_conj3}.
\end{proof}

\begin{definition}

\noindent{(1)} Any convex polygon $P$ in $\mathbb{R}^2$ has one or two rightmost vertices.   Of those, the vertex whose $y$-coordinate is higher is called the \emph{most East-North vertex} (or \emph{EN vertex} for short) of $P$.

\noindent{(2)} Any convex polygon $P$ in $\mathbb{R}^2$ has one or two uppermost vertices.   Of those, the vertex whose $x$-coordinate is larger is called the \emph{most North-East vertex} (or \emph{NE vertex} for short) of $P$.

\end{definition}

A very important step towards proving Conjecture~\ref{Jac_conj4} is to prove the following conjecture(s):
\begin{conjecture}\label{weak_Jac_conj}
(1)  The Newton polygon of the innermost polynomial $Z$ has the following property: its NE vertex $(m',n')$, which coincides with its EN vertex, satisfies $n'/m'=n/m$ and $N^0(Z)\subseteq \text{Rect}_{m',n'}$. 

\noindent (2) The Newton polygon of the inner polynomial $\zeta(F)$ has the same property. 

 \end{conjecture}

The NE vertex of $N^0(Z)$ (resp. $N^0(\zeta(F))$) is called the \emph{innermost vertex} (resp. \emph{inner vertex}). 
Conjecture~\ref{weak_Jac_conj}(1) (resp. Conjecture~\ref{weak_Jac_conj}(2)) is called the innermost vertex conjecture (resp. inner vertex conjecture).

The main theorem of this paper is as follows.

\begin{theorem}\label{main_thm}
Assume $m\le n$. Assume that  $F\in R_{m,n}$, $G\in R_{bm/a, bn/a}$, and $[F,G]\in \mathbb{C}$.  
Let $Z$ be the innermost  polynomial of $F$.  Then either $Z=0$, or
there exists a lattice point $(m',n')$ such that

\noindent\emph{(1)} $Z\in \overline{R}_{m',n'}$; 

\noindent\emph{(2)} $(m',n')\in \mathfrak{R} := \{(x,y)\in \mathbb{R}^2 \, : \, 0\le y< \frac{a-1}{a}n\,\text{ and }\,\mathfrak{m}(x-\frac{a}{a+b}) +\frac{a}{a+b}\le y\le \frac{n}{m}x \}$, where $\mathfrak{m}=(n-\frac{a}{a+b})/(m-\frac{a}{a+b})$ is the slope of the line through $(m,n)$ and $C:=(\frac{a}{a+b},\frac{a}{a+b})$.
\end{theorem}

\begin{remark}\label{remark:narrow region}
The lattice points in $\mathfrak{R}$  are in a very narrow region as drawn in Figure \ref{fig:narrow region}. Given two lattice points $(x_1,y_1)$ and $(x_2,y_2)$ in this region, $x_1<x_2$ implies $y_1<y_2$,  because the width of the region at any height is $<\frac{a}{a+b}<1$. So there are at most $n(a-1)/a-1$ lattice points in the region.

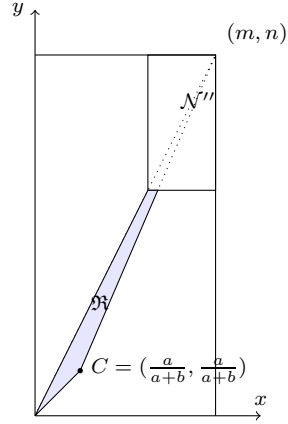
\begin{figure}[h]
\begin{tikzpicture}[scale=0.60]
\draw (0,8)--(4,8)--(4,0);
\draw[dotted] (0,0)--(4,8)--(2.72,5);
\draw (2.5,5)--(2.5,8)--(4,8)--(4,5)--(2.5,5);
\draw[fill=blue!10] (0,0)--(1,1)--(2.72,5)--(2.5,5)--(0,0);
\draw (3,7) node[anchor=west] {\tiny $\mathcal{N}''$};
\draw (1,2.5) node[anchor=west] {\tiny $\mathfrak{R}$};
\draw (1,1) node {\huge .};
\draw (1,1) node[anchor=west] {\tiny $C=(\frac{a}{a+b},\frac{a}{a+b})$};
\draw (4, 8) node[anchor=south west] {\tiny $(m,n)$};
\draw[->] (0,0) -- (5,0)
node[above] {\tiny $x$};
\draw[->] (0,0) -- (0,9)
node[left] {\tiny $y$};
\end{tikzpicture}
\caption{The narrow (shaded) region}
\label{fig:narrow region}
\end{figure}
\end{remark}

\begin{corollary}\label{main_cor}
Theorem \ref{main_thm} still holds if $Z$ is replaced by the inner polynomial $\zeta(F)$ of $F$. \end{corollary}
\begin{proof}
We have $\zeta(F) = Z + \sum_{j=0}^{(a-1)\delta-1} e_jW_F^j$. If all $e_j$ are $0$, then $\zeta(F)=Z$, and the statement trivially holds. Otherwise, let $j$ be the largest integer such that $e_j\neq0$. Note that $(jm/a\delta,jn/a\delta)$ is the NE vertex of $N^0(W_F^j)$, and  $(jm/a\delta,jn/a\delta)\neq(m',n')$ by the definition of $Z$ in \eqref{eq:Z}.  
By Remark \ref{remark:narrow region}, $(jm/a\delta,jn/a\delta)$ may lie to the southwest, or to the northeast (or north) of $(m',n')$. 
In the former case, $\zeta(F)$ is in $\overline{R}_{m',n'}$, thus (1) and (2) still hold; 
in the latter case, we see that $\zeta(F)$ is in $\overline{R}_{jm/a\delta,jn/a\delta}$ , and $(jm/a\delta,jn/a\delta)\in\mathfrak{R}$, so (1) and (2) hold after replacing $(m',n')$ by $(jm/a\delta,jn/a\delta)$.
\end{proof}

\begin{corollary}
Conjecture \ref{weak_Jac_conj} holds for the case that $\frac{e}{d}$ is sufficiently large, more precisely, if $\frac{e}{d}>\frac{n-m}{a}-1$. 
\end{corollary}
\begin{proof}
The idea of the proof is that when $\frac{b}{a}=\frac{e}{d}\to\infty$, the point $C\to(0,0)$, thus $\mathfrak{R}$ gets narrower, and eventually all integer points in $\mathfrak{R}$ must lie on the diagonal of $\Rect_{m,n}$. 
More precisely, assume $P$ is an integer point in $\mathfrak{R}$ but not on the diagonal of $\Rect_{m,n}$.  The area of $OCP$ is less than or equal to 
the area of the triangle $OC(m,n)$, which is $\frac{1}{2}\begin{vmatrix}a/(a+b)&a/(a+b)\\m&n\end{vmatrix}=\frac{a(n-m)}{2(a+b)}$. 
But Pick's theorem, the area of the triangle $OCP$ is $i+\frac{j}{2}-1$, where $i$ is the number of interior integer points, $j$ is the number of integer points on the boundary.  Since $i\ge0$, $j\ge\gcd(m,n)+2$, we have
$\frac{a(n-m)}{2(a+b)}\ge i+\frac{j}{2}-1 \ge \frac{\gcd(m,n)}{2} \ge \frac{a}{2}$, thus $\frac{e}{d}=\frac{b}{a} \le \frac{n-m}{a}-1$. 
\end{proof}

\section{Proof of the main theorem}

\subsection{A useful lemma}

Recall that for a polynomial $f(s)$ in one variable,  $\deg f $ is the highest of the degrees of $f$'s monomials with non-zero coefficients; Analogously,  we denote ${\rm lowdeg\,} f$ to be the lowest of the degrees of $f$'s monomials with non-zero coefficients. For example,
$\deg(2x^3+4x^5)=5$ and ${\rm lowdeg\,} (2x^3+4x^5)=3$. 


Recall the definition of the (generalized) binomial coefficients:
for $c\in\mathbb{R}$, $i\in \mathbb{Z}_{\ge0}$, denote the binomial coefficient
$\binom{c}{i}=\frac{c(c-1)(c-2)\cdots(c-i+1)}{i!}$. 
Let $A,B\in\mathbb{C}[s]$ such that the constant term of $A$ is nonzero. Then $A^{-1}$ is well-defined in $\mathbb{C}[[s]]$.  
For $a\in\mathbb{Z}_{>0}$, $c\in\mathbb{Q}$ such that $ac\in\mathbb{Z}$, define the rational power expansion 
$(A^a+B)^{c} = \sum_{i=0}^\infty {c\choose i} A^{a(c-i)}B^i$.  
(See \S\ref{section:Appendix B} for the definition of multinomial coefficients and a similar expansion.) 

\begin{lemma}\label{I3I2U5U4old}
Let $A,B\in \mathbb{C}[s]$. Suppose that the constant term of $A$ is nonzero, $\deg B < a({\rm lowdeg\,} B)$, and $\deg A< {\rm lowdeg\,} B$.
For each $\mu\in\mathbb{Z}_{\ge 0}$, let $$R_\mu= \bigg[(A^a + B)^{b/a}\bigg]_{s^{\mu}} = \bigg[\sum_{i=0}^\infty\binom{b/a}{i}A^{b-ai}B^i\bigg]_{s^{\mu}}.$$  
If $R_\mu=0$ for all sufficiently large $\mu$, then $B=0$. In particular, if $(A^a + B)$ is the $a$-th power of an element of $\mathbb{C}[s]$,  then $B=0$. 
\end{lemma}
\begin{proof}
Assume the contrary that $R_\mu=0$ for all sufficiently large $\mu$ and $B\neq 0$. Without loss of generality we can assume the constant term of $A$ is $1$. 
Write $A=1+ a_1 s + \cdots + a_j s^j$ (where $a_j$ may be $0$) and $B=b_{j+1} s^{j+1}+\cdots + b_{k} s^{k}$ (where $b_{j+1}, b_{k}\neq0$, and $j+1\le k$). 
Since $R_\mu=0$ for all sufficiently large $\mu$, $(A^a + B)^{b/a}$ must be a polynomial. Then $A^a+ B=C^a$ for some polynomial $C=1+ c_1 s + \cdots + c_\ell s^\ell \in  \mathbb{C}[s]$ with $c_\ell\neq0$. Note that $\ell=\deg C\le \max(\deg A, (\deg B)/a)\le \max(j,k/a)$. 
Write $B=C^a-A^a=(C-A)(C^{a-1}+C^{a-2}A+\cdots+A^{a-1})$. Since the constant term of $(C^{a-1}+C^{a-2}A+\cdots+A^{a-1})$ is $a\neq0$, we have ${\rm lowdeg\,} (C-A) = {\rm lowdeg\,} B=j+1$. On the other hand, 
$\deg (C-A)\le \max(\deg C,\deg A)\le (\ell,j)\le \max(j,k/a)$. By ${\rm lowdeg\,} (C-A)\le \deg(C-A)$, we have
$j+1\le \max(j, k/a)$, so  $k\ge a(j+1)$, that is,  $\deg B\ge a({\rm lowdeg\,} B)$, a contradiction. 
\end{proof}

\subsection{Proof of Theorem~\ref{main_thm} }

\begin{definition}
 
Let $\mathcal{U}^{(1,2)}_p=\mathbb{C}[x^{\pm 1/p}, y]$ for $p\in\mathbb{Z}_{>0}$, and let $\mathcal{U}^{(1,2)}=\cup_{p\in\mathbb{Z}_{>0}} \mathcal{U}^{(1,2)}_p$ as a subring of $\mathbb{C}\langle\langle x\rangle\rangle[y]$, where $\mathbb{C}\langle\langle x\rangle\rangle=\{ \sum_{i\ge r} a_ix^{i/n} \ | \  r,n\in\mathbb{Z}, n\ge1\}$ (see \cite[\S1.2]{CA}). 

Similarly, 
Let $\mathcal{U}^{(1,4)}_p=\mathbb{C}[x,y^{\pm 1/p}]$ for $p\in\mathbb{Z}_{>0}$, and let $\mathcal{U}^{(1,4)}=\cup_{p\in\mathbb{Z}_{>0}} \mathcal{U}^{(1,4)}_p$. 
\end{definition}
In the above definition,  (1,2) stands for the first and second quadrants, and (1,4) stands for the first and fourth quadrants.

\begin{definition}
Suppose that $w=(v,-u)\in \mathcal{D}$ with coprime integers $u\ge0$ and $v>0$.  Let $f\in \mathcal{U}^{(1,2)}$. Suppose that $f_+^w = \prod_i (x^{u/v} y - \alpha_i)^{m_i}$. 
An automorphism $\phi\in\text{Aut}(\mathcal{U}^{(1,2)})$ is called \emph{decreasing} with respect to $f$ if $\phi(x)=x$ and $\phi(y)=y + \alpha_i x^{-u/v}$ for some $i$. This is called decreasing because if all $\alpha_i$ are nonzero, then  $\phi$ decreases the area of $N^0(F)\cap (\text{the first quadrant})$. For example, the maps in Figures~\ref{one_root} and ~\ref{two_roots} are decreasing. 
\end{definition}

As in \cite[Definition 1.5]{GGV}, we give an order on directions as follows.

\begin{definition}\label{definition:direction order}
We consider the map $\mathcal{D}\to S^1$ by assigning to each direction its corresponding unit vector in $S^1$, 
and we define an \emph{interval} in $\mathcal{D}$ as the preimage 
under this map of an arc of $S^1$ that is not the whole circle.
We consider each interval endowed with the order that increases
counterclockwise. 
\end{definition}

\begin{definition}
Fix $(a,b,m,n)\in\mathcal{Q}$. Let $C=(\frac{a}{a+b},\frac{a}{a+b})$.
For any point $v\in \mathbb{Q}_{>\frac{a}{a+b}}\times \mathbb{Z}_{>0}$, let $\mathcal{L}_v$ be the intersection of the first quadrant with the right side of the line through $v$ and $C$, that is, 
$$\mathcal{L}_v=\{(x,y)\in \mathbb{R}_{\ge 0}^2 \, : \, 0<y< \mathfrak{m}_{v}(x-\frac{a}{a+b}) +\frac{a}{a+b} \},$$
where $\mathfrak{m}_{v}>0$ is the slope of the line through $v$ and $C$. 
\end{definition}

\begin{proposition}\label{no12v2}
Let $(a, b, m, n) \in \mathcal{Q}$ and assume $m\le n$. 
Suppose that $\ff,\GG\in \mathbb{C}[x,y]$ satisfy the following:

\noindent\emph{(1)} $[\ff,\GG]  \in  \mathbb{C}$;

\noindent\emph{(2)} $\ff\in R_{m,n}$ and $\GG\in R_{bm/a,bn/a}$.

\noindent Let $Z$ be the innermost polynomial of $F$. Then the EN vertex of $N^0(Z)$ must be in $$\{(x,y)\in \mathbb{R}_{\ge0}^2 \, : \, y\ge \mathfrak{m}_{(m,n)}(x-\frac{a}{a+b}) +\frac{a}{a+b} \}.$$ 
\end{proposition}

\begin{proof}
If $Z=0$, then $N^0(Z)$ is the origin and the statement obviously holds. So we assume that $Z\neq0$ for the rest of the proof.

Let $\mathcal{T}$ be the set of 5-tuples $(f,g,z,v_z,v_f)$ of three Laurent polynomials $f,g,z \in \mathcal{U}^{(1,2)}$ and two points $v_z,v_f\in \mathbb{Q}\times \mathbb{Z}_{>0}$ satisfying the following:

\noindent{(a)} $[f,g]\in\mathbb{C}$ and $f-z=\mathfrak{Q}^a+\sum_{i=0}^{(a-1)\delta-1} e_{i}\mathfrak{W}^{i}$, where $\mathfrak{W} \in \mathcal{U}^{(1,2)}$, and $\mathfrak{Q}$ is a degree $\delta$ polynomial of $\mathfrak{W}$;

\noindent{(b)} $v_z$ is a vertex of $N^0(z)$ and $v_f$ is a vertex of $N^0(f)$;

\noindent{(c)} $\text{deg}_y(v_z)< \frac{a-1}{a} \text{deg}_y(v_f)$;

\noindent{(d)} $v_z\in \mathcal{L}_{v_f} \cup \{(x,0)\in \mathbb{R}^2 \, : \, x>0\}$.

Aiming at contradiction, suppose that the EN vertex of $N^0(Z)$ is in $$\mathcal{L}_{(m,n)} \cup \{(x,0)\in \mathbb{R}^2 \, : \, x>0\}.$$ 

\noindent Step 1. Construct a finite sequence of 5-tuples $(F^{(j)}, G^{(j)}, Z^{(j)}, v_{Z^{(j)}}, v_{F^{(j)}})\in \mathcal{T}$.

Let $p_0=1$. 
Let $F^{(0)}=F$,  $G^{(0)}=G$, and $Z^{(0)}=Z=F-Q^a - \sum_{i=0}^{(a-1)\delta-1} e_{i}\EEE^{i}$.  Let $v_{Z^{(0)}}$ be the EN vertex of $N^0(Z)$, and $v_{F^{(0)}}$ be the EN vertex of $N^0(F)$, which is $(m,n)$. Then $$(F^{(0)},  G^{(0)}, Z^{(0)}, v_{Z^{(0)}}, v_{F^{(0)}})\in \mathcal{T}.$$
We will inductively construct a finite sequence of 5-tuples $(F^{(j)}, G^{(j)}, Z^{(j)}, v_{Z^{(j)}}, v_{F^{(j)}})\in \mathcal{T}$ for $j\ge 0$ as follows.
Suppose that the 5-tuple for $j$ has been constructed. 
Let $v'_{F^{(j)}}$ be the vertex of $N^0(F^{(j)})$ uniquely determined by the condition that $v_{F^{(j)}}v'_{F^{(j)}}$ is an edge of $N^0(F^{(j)})$ and if 
$v_{F^{(j)}}v''_{F^{(j)}}$ is the other edge, then its outward normal direction is larger than the outward normal direction of $v_{F^{(j)}}v'_{F^{(j)}}$  (in the sense of Definition \ref{definition:direction order}).
It is easy to see that $\deg_y(v'_{F^{(j)}})<\deg_y(v''_{F^{(j)}})$. 
Let $w_j=(v_j,-u_j)$ be the outward normal direction of the edge $v_{F^{(j)}}v'_{F^{(j)}}$,  where $u_j\ge0$ and $v_j>0$ are coprime integers. Given a point $P$ and a direction vector $w$, denote ${\rm Line}_w(P)$ to be the line passing through $P$ that is normal to $w$.

We stop the inductive process if the following condition fails: 

\noindent ($\star_j$) \quad\quad\quad The point $v(Z^{(j)})$ lies strictly to the left of ${\rm Line}_{w_j}(v_{F^{(j)}})$.

\noindent Note that ($\star_0$)  holds since $w_j=(1,0)$, and $v(Z^{(0)})$ is strictly to the left of ${\rm Line}_{w_j}(v_{F^{(0)}})=\{x=m\}$. 
 
Suppose ($\star_j$) holds. We can write 
$$(F^{(j)})_+^{w_j} =c_j x^{\text{deg}_{w_j}(F^{(j)})/v_j}  \prod_{i=1}^{\mathfrak{u}_j} (x^{u_j/v_j} y - \alpha_{ji})^{m_{ji}},$$
where $\alpha_{ji}\in\mathbb{C}$ ($1\le j\le \mathfrak{u}_j$) are all distinct, 
$\mathfrak{u}_j, m_{j1},...,m_{j\mathfrak{u}_j}\in\mathbb{Z}_{>0}$, and $c_j\in\mathbb{C}\setminus\{0\}$. 
We can write 
$$(Z^{(j)})_+^{w_j}= c'_j x^{\text{deg}_{w_j}(Z^{(j)})/v_j} \bigg(\prod_{i=1}^{\mathfrak{u_j}} (x^{u_j/v_j} y - \alpha_{ji})^{n_{ji}}\bigg)\bigg(\prod_{i=\mathfrak{u}_j+1}^{\mathfrak{v}_j} (x^{u_j/v_j} y - \alpha_{ji})^{n_{ji}}\bigg)$$
where $\alpha_{ji}\in\mathbb{C}$ ($1\le j\le \mathfrak{v}_j$) are all distinct, $\alpha_{j1},\dots,\alpha_{j\mathfrak{u}_j}$ are those occurred in the expression of $(F^{(j)})_+^{w_j}$, $\mathfrak{v}_j\ge\mathfrak{u}_j$, $n_{j1},...,n_{j\mathfrak{v}_j}\in\mathbb{Z}_{\ge0}$, and $c'_j\in\mathbb{C}\setminus\{0\}$. 
Note that $\text{deg}_y(v_{F^{(j)}})=\sum_{i=1}^{\mathfrak{u}_j} m_{ji}$, 
$\text{deg}_y(v_{Z^{(j)}})=\sum_{i=1}^{\mathfrak{v}} n_{ji}$.
Since $(F^{(j)}, G^{(j)}, Z^{(j)}, v_{Z^{(j)}}, v_{F^{(j)}})\in \mathcal{T}$, we have 
\begin{equation}\label{ineq77}
\sum_{i=1}^{\mathfrak{u}_j} n_{ji}<\frac{a-1}{a}\sum_{i=1}^{\mathfrak{u}_j} m_{ji},\textrm{ or equivalently, }\frac{\text{deg}_y(v_{Z^{(j)}})}{\text{deg}_y(v_{F^{(j)}})}<\frac{a-1}{a}.
\end{equation}

Choose an integer $\mathfrak{i}=\mathfrak{i}_j\in \{1,\dots, \mathfrak{u}_j\}$ such that the following holds (for example, we can take $\mathfrak{i}$ such that $\frac{n_{j\mathfrak{i}}}{m_{j\mathfrak{i}}}$ is minimal): 
\begin{equation}\label{ineq70707}
\frac{n_{j\mathfrak{i}}}{m_{j\mathfrak{i}}} \le \frac{\sum_{i=1}^{\mathfrak{u}_j} n_{ji}}{\sum_{i=1}^{\mathfrak{u}_j} m_{ji}},
\textrm{ or equivalently, }
\frac{\text{deg}_y(v_{Z^{(j+1)}})}{\text{deg}_y(v_{F^{(j+1)}})}
\le 
\frac{\text{deg}_y(v_{Z^{(j)}})}{\text{deg}_y(v_{F^{(j)}})}
\end{equation} 
where $v_{Z^{(j+1)}}$ and $v_{F^{(j+1)}}$ are defined as follows. 

Define $$p_j={\rm lcm}(v_j, p_{j-1})\in\mathbb{Z}_{>0}.$$ 
Then  $\varphi_j:\mathcal{U}^{(1,2)}_{p_j}\longrightarrow  \mathcal{U}^{(1,2)}_{p_j}$ given by $x\mapsto x$ and $y\mapsto y + \alpha_\mathfrak{i}x^{-u_j/v_j}$ is a well-defined decreasing automorphism of  $\mathcal{U}^{(1,2)}_{p_j}$  with respect to $F^{(j)}$. 
Let $F^{(j+1)}=\varphi_j(F^{(j)})$, $G^{(j+1)}=\varphi_j(G^{(j)})$, and $Z^{(j+1)}=\varphi_j(Z^{(j)})$. 
Let $v_{Z^{(j+1)}}$ be the lower vertex of the  edge of $N^0(Z^{(j+1)})$ with outward normal vector $w_j$. (We allow the edge to have length 0, in which case $v^{Z^{(j+1)}}=v^{Z^{(j)}}$.)   
Similarly, let  $v_{F^{(j+1)}}$ be the lower vertex of the  edge of $N^0(F^{(j+1)})$ with outward normal vector $w_j$.  (We allow the edge to have length 0, in which case $v_{F^{(j+1)}}=v_{F^{(j)}}$.)
Note that $\deg_y(v_{Z^{(j+1)}}) = n_\mathfrak{i}$,  $\deg_y(v_{F^{(j+1)}}) = m_\mathfrak{i}$. 
Also note that
$v_{F^{(j+1)}}=v_{F^{(j)}}$ if and only if $\mathfrak{u}_j=1$ (so $\mathfrak{i}=1$); 
$v_{Z^{(j+1)}}=v_{Z^{(j)}}$ if and only if $n_j=0$ for all $j\neq\mathfrak{i}$. 
(See Figure~\ref{one_root} for an example of the case where $v_{Z^{(j+1)}}=v_{Z^{(j)}}$ and $v_{F^{(j+1)}}=v_{F^{(j)}}$;
see Figure~\ref{two_roots} for an example of the case where $v_{Z^{(j+1)}}\neq v_{Z^{(j)}}$ and $v_{F^{(j+1)}}\neq v_{F^{(j)}}$.)

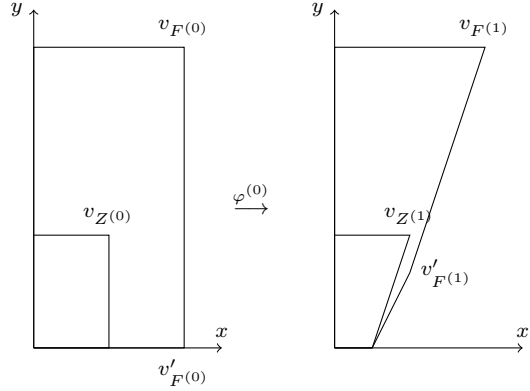
\begin{figure}
\begin{center}
\begin{tikzpicture}[scale=0.5]
\usetikzlibrary{patterns}
\draw (0,0)--(4,0)--(4,8)--(0,8)--(0,0);
\draw (0,0)--(2,0)--(2,3)--(0,3);
\draw[->] (0,0) -- (5,0)
node[above] {\tiny $x$};
\draw[->] (0,0) -- (0,9)
node[left] {\tiny $y$};
\draw (2,3) node[anchor=south] {\tiny $v_{Z^{(0)}}$};
\draw (4,8) node[anchor=south] {\tiny $v_{F^{(0)}}$};
\draw (4,0) node[anchor=north] {\tiny $v'_{F^{(0)}}$};
\draw (5,4) node[anchor=west]{\tiny $\overset{\varphi^{(0)}}{\longrightarrow}$};
\begin{scope}[shift={(8,0)}]
\draw (0,0)--(1,0)--(2,2)--(4,8)--(0,8)--(0,0);
\draw (0,0)--(1,0)--(2,3)--(0,3);
\draw (2,3) node[anchor=south] {\tiny $v_{Z^{(1)}}$};
\draw (4,8) node[anchor=south] {\tiny $v_{F^{(1)}}$};
\draw (2,2) node[anchor=west] {\tiny $v'_{F^{(1)}}$};
\draw[->] (0,0) -- (5,0)
node[above] {\tiny $x$};
\draw[->] (0,0) -- (0,9)
node[left] {\tiny $y$};
\end{scope}
\end{tikzpicture}
\end{center}
	\caption{If $Z_+^{(1,0)}=x^2(y-\alpha_{01})^3$ and  $F_+^{(1,0)}=x^4(y-\alpha_{01})^8$, then $v_{Z^{(0)}}=v_{Z^{(1)}}=(2,3)$ and  $v_{F^{(0)}}=v_{F^{(1)}}=(4,8)$ }
\label{one_root}
\end{figure}

\begin{figure}[h]
\begin{center}
\begin{tikzpicture}[scale=0.5]
\usetikzlibrary{patterns}
\draw (0,0)--(4,0)--(4,8)--(0,8)--(0,0);
\draw (0,0)--(2,0)--(2,3)--(0,3);
\draw (2,3) node[anchor=south] {\tiny $v_{Z^{(0)}}$};
\draw (4,8) node[anchor=south] {\tiny $v_{F^{(0)}}$};
\draw (4,0) node[anchor=north] {\tiny $v'_{F^{(0)}}$};
\draw[->] (0,0) -- (5,0)
node[above] {\tiny $x$};
\draw[->] (0,0) -- (0,9)
node[left] {\tiny $y$};
\draw(6,4) node {\tiny$\overset{\varphi^{(0)}}{\longrightarrow}\quad$};
\begin{scope}[shift={(8,0)}]
\draw (0,0)--(3,0)--(4,4)--(4,8)--(0,8)--(0,0);
\draw (0,0)--(1,0)--(2,1)--(2,3)--(0,3);
\draw (2,3) node[anchor=south] {\tiny $v_{Z^{(0)}}$};
\draw (4,8) node[anchor=south] {\tiny $v_{F^{(0)}}$};
\draw (2,1) node[anchor=east] {\tiny $v_{Z^{(1)}}$};
\draw (4,4) node[anchor=west] {\tiny $v_{F^{(1)}}$};
\draw (3.3,0) node[anchor=north] {\tiny $v'_{F^{(1)}}$};
\draw[->] (0,0) -- (5,0)
node[above] {\tiny $x$};
\draw[->] (0,0) -- (0,9)
node[left] {\tiny $y$};
\end{scope}
\end{tikzpicture}
\end{center}
	\caption{If $Z_+^{(1,0)}=x^2(y-\alpha_{01})(y-\alpha_{02})^2$ and  $F_+^{(1,0)}=x^4(y-\alpha_{01})^4 (y-\alpha_{02})^4$, then $(2,3)=v_{Z^{(0)}}\neq v_{Z^{(1)}}=(2,1)$ and  $(4,8)=v_{F^{(0)}}\neq v_{F^{(1)}}=(4,4)$ }
\label{two_roots}
\end{figure}
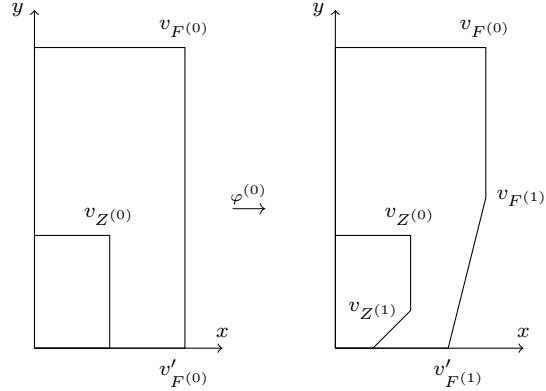

It is straightforward to check that $(F^{(j+1)}, G^{(j+1)}, Z^{(j+1)}, v_{Z^{(j+1)}}, v_{F^{(j+1)}})$ is in $\mathcal{T}$ as follows: 

\noindent (a) is obvious, because $\varphi_j$ is an automorphism. 

\noindent (b) is clear from the construction. 

\noindent (c) follows from \eqref{ineq70707} and \eqref{ineq77}.

\noindent (d)  
If $Z^{(j)}$ is already on the $x$-axis,  then $Z^{(j+1)}=Z^{(j)}$, the condition (d) trivially holds. So we can assume
$\deg_y(Z^{(j)})>0$.
Define $D_j$ to be the point on the $x$-axis that is colinear with $v_{Z^{(j)}}$ and $v_{F^{(j)}}$, and define  $D_{j+1}$ similarly. 
Let $E$ be the intersection of the segment $D_jv_{F^{(j+1)}}$ with ${\rm Line}_w(v_{Z^{(j)}})$. 
Condition ($\star_j$) implies that $v_{F^{(j+1)}}$ and $E$ are strictly on the right side of the line $D_jv_{F^{(j)}}$. Using \eqref{ineq70707} and the similarity of the triangles $D_jv_{Z^{(j)}}E$ and $D_jv_{F^{(j)}}v_{F^{(j+1)}}$, we have 
$$\frac{\text{deg}_y(v_{Z^{(j+1)}})}{\text{deg}_y(v_{F^{(j+1)}})}
\le 
\frac{\text{deg}_y(v_{Z^{(j)}})}{\text{deg}_y(v_{F^{(j)}})}
=\frac{|D_jv_{Z^{(j)}}|}{|D_jv_{F^{(j)}}|}
=\frac{|D_jE|}{|D_jv_{F^{(j+1)}}|}
=\frac{\text{deg}_yE}{\text{deg}_y(v_{F^{(j+1)}})}.
$$
So $\text{deg}_y(v_{Z^{(j+1)}})\le\text{deg}_yE$, which implies the condition (d) that $v_{Z^{(j+1)}}$ is in $\mathcal{L}_{v_{F^{(j+1)}}} \cup \{(x,0)\in \mathbb{R}^2 \, : \, x>0\}$. 
See Figure~\ref{Tdef4} for illustration.
So $(F^{(j+1)}, G^{(j+1)}, Z^{(j+1)}, v_{Z^{(j+1)}}, v_{F^{(j+1)}})$ is in $\mathcal{T}$.

\begin{figure}[h]
\begin{center}
\begin{tikzpicture}[scale=0.5]
\usetikzlibrary{patterns}
\draw (6,3)--(5,1);
\draw (12,8)--(10,4);
\draw (6,3) node {\huge .};
\draw (12,8) node {\huge .};
\draw (5,1) node {\huge .};
\draw (10,4) node {\huge .};
\draw (6,3) node[anchor=west] {\tiny $v_{Z^{(j)}}$};
\draw (12,8) node[anchor=south] {\tiny $v_{F^{(j)}}$};
\draw (5,1) node[anchor=west] {\tiny $v_{Z^{(j+1)}}$};
\draw (10,4) node[anchor=west] {\tiny $v_{F^{(j+1)}}$};
\draw[->] (0,0) -- (15,0)
node[above] {\tiny $x$};
\draw[->] (0,0) -- (0,9)
node[left] {\tiny $y$};
\draw (1,1) node {\huge .};
\draw (1,1) node[anchor=south] {\tiny $C=(\frac{a}{a+b},\frac{a}{a+b})$};
\draw (0,0) node[anchor=east] {\tiny $O=(0,0)$};
\draw (2.5,0) node[anchor=north] {\tiny $D_j$};
\draw (4,0) node[anchor=north] {\tiny $D_{j+1}$};
\draw[dotted] (1,1)--(12,8);
\draw[densely dotted] (12,8)--(6,3)--(6-6*3/5,0);
\draw[densely dotted] (10,4)--(5,1)--(5-5*1/3,0);
\draw [-stealth](11,6) -- (12,5.5);
\draw (12,5.5) node[anchor=west] {\tiny $w$};
\draw[dotted] (2.5, 0)--(10,4);
\draw (5.25,1.5) node {\huge .};
\draw (5,1.5) node[anchor=south] {\tiny $E$};
\end{tikzpicture}
\end{center}
	\caption{Illustration for the proof of (d). }
\label{Tdef4}
\end{figure}
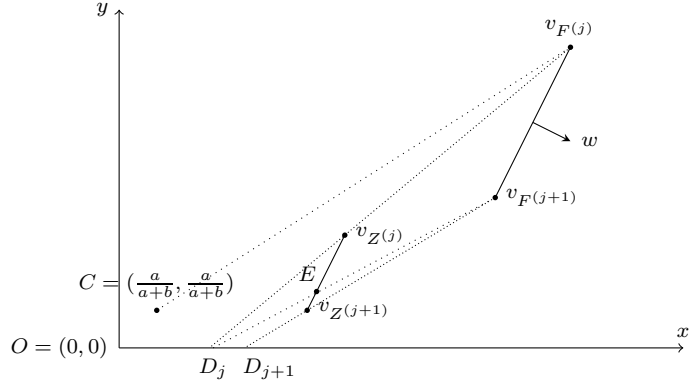

Next, we show that the above inductive process must stop in finitely many steps. We show this by contradiction and assume the process will not stop. Since $\deg_y(v_{F^{(j)}})\in\mathbb{Z}_{>0}$ and $\deg_y(v_{Z^{(j)}})\in\mathbb{Z}_{\ge0}$ for all $j$, there exists $j_0$ such that 
$v_{F^{(j)}}=v_{F^{(j_0)}}=(c,d)$ for all $j\ge j_0$. 
Then for each $j\ge j_0$, 
$(F^{(j+1)})_+^{w_{j+1}}=c_{j+1} x^{\text{deg}_{w_{j+1}}(F^{(j+1)})/v_{j+1}}  (x^{u_{j+1}/v_{j+1}} y-\alpha_{j+1,1})^d$ which must be in $\mathcal{U}^{(1,2)}_{p_j}$. 
Since $c_{j+1}$ and $\alpha_{j+1,1}$ are nonzero constants, $x^{\text{deg}_{w_{j+1}}(F^{(j+1)})/v_{j+1}}  (x^{u_{j+1}/v_{j+1}} y)^i$ (for $0\le i\le d$) must be in  $\mathcal{U}^{(1,2)}_{p_j}$.
It follows that $x^{u_{j+1}/v_{j+1}} y$ must be in $\mathcal{U}^{(1,2)}_{p_j}$. 
This implies $v_{j+1}|p_j$, thus $p_{j+1}={\rm lcm}(v_{j+1},p_j)=p_j$. So $p_j=p_{j_0}$ for all $j\ge j_0$. But then we obtain a contradiction that infinitely many distinct points $v'_{F^{(j)}}$ ($j\ge j_0$) are in the finite set $(\frac{1}{p_{j_0}}\mathbb{Z}\times\mathbb{Z})\cap ([0,m]\times[0,n])$.

\noindent Step 2. Assume the inductive process stops at $j$. So ($\star_{j}$)  fails. For simplicity, denote $w=w_{j}$. By induction it is easy to see that $v_{G^{(j')}} =\frac{b}{a}v_{F^{(j')}}$ for $0\le j'\le j$. So $$\deg_w(G^{(j')})=\frac{b}{a}\deg_w(F^{(j')})\textrm{ for } 0\le j'\le j.$$ 
Now we consider two cases separately.

\noindent Case 1:  $v(Z^{(j)})$ lies on ${\rm Line}_{w}(v_{F^{(j)}})$. 

By \eqref{ineq77}, $\deg_y v(Z^{(j)})<\deg_y v(F^{(j)})$, so  
$(F^{(j)})_+^w = (F^{(j)}-Z^{(j)})_+^w + (Z^{(j)})_+^w$, where $(F^{(j)}-Z^{(j)})_+^w=((Q^{(j)})_+^w)^a$ is the $a$-th power of  $(Q^{(j)})_+^w\in \mathcal{U}^{(1,2)}$. Since $v_{Z^{(j)}}$ is in $\mathcal{L}_{v_{F^{(j)}}}$, we have 
$$\deg_w( F^{(j)} ) > \deg_w(C).$$
If $[(F^{(j)})_+^w, (G^{(j)})_+^w]\neq 0$, then 
$$
\aligned
\deg_w( [ F^{(j)}, G^{(j)} ] ) &= \deg_w( F^{(j)} ) + \deg_w( G^{(j)} ) -\deg_w((1,1))\\
&=\frac{a+b}{a}\deg_w( F^{(j)} ) -\deg_w((1,1)) \\
&>\frac{a+b}{a}\deg_w( C) -\deg_w((1,1))
=\deg_w(\frac{a+b}{a}C) -\deg_w((1,1))
=0,
\endaligned
$$
which contradicts to $[F^{(j)}, G^{(j)}]\in\mathbb{C}$. So $[(F^{(j)})_+^w, (G^{(j)})_+^w]=0$, which in turn implies that  $(F^{(j)})_+^w$ is the $a$-th power of an element in $\mathcal{U}^{(1,2)}$.  
Let $v_{F^{(j)}}=(c,d)$.
Let $\mathcal{R}'=\mathbb{C}$,  let $A(s)$ and $B(s)$ be polynomials uniquely determined by 
$$
(Q^{(j)})_+^w = x^{c/a}y^{d/a}A(x^{-u_{j}/v_{j}}),\quad
(Z^{(j)})_+^w = x^{c}y^{d}B(x^{-u_{j}/v_{j}}).
$$
Then  $\deg A(s)\le d/a$, $\deg B(s)\le d$, ${\rm lowdeg}B(s)=\deg_y(v_{F^{(j)}})-\deg_y(v_{Z^{(j)}})>d/a$  by \eqref{ineq77}.
Thus $\deg B(s) < a({\rm lowdeg\,} B(s))$, and $\deg A(s)< {\rm lowdeg\,} B(s)$, and we can apply Lemma~\ref{I3I2U5U4old} to conclude that $B(s)=0$, which contradicts the fact that $(Z^{(j)})_+^w\neq0$. 

\noindent Case 2:  $v(Z^{(j)})$ lies strictly on the right side of ${\rm Line}_{w}(v_{F^{(j)}})$. 

In this case, 
consider the line passing through $v_{F^{(j)}}$ and  $v_{Z^{(j)}}$ and denote its rightward normal vector by $w'$. 
Denote $v_{F^{(j)}}=(c,d)$, $v_{Z^{(j)}}=(c',d')$.  Then
$(F^{(j)})_+^{w'} = (F^{(j)}-Z^{(j)})_+^{w'} + (Z^{(j)})_+^{w'}$  implies
$c_jx^cy^d-c'_jx^{c'}y^{d'}=(F^{(j)}-Z^{(j)})_+^{w'}=((Q^{(j)})_+^{w'})^a$ is the $a$-th power of  $(Q^{(j)})_+^{w'}\in \mathcal{U}^{(1,2)}$, which is impossible since $c_j, c_j'\neq0$.
\end{proof}

\begin{example}
Proposition~\ref{no12v2} is equivalent to saying that if the EN vertex of $N^0(Z)$ is in $$\mathcal{L}_{(m,n)} \cup \{(x,0)\in \mathbb{R}^2 \, : \, x>0\},$$ then $[F,G]$ is not a constant. We illustrate this by the following example.
Let $$\aligned 
&(a,b,m,n)=(2,3,2,4),\\
&F=F^{(0)}=(x+1)^2(y+1)^4 + (x+1)(y+1) \in {R}_{2,4},\\
&G=G^{(0)}=(x+1)^3(y+1)^6 + \frac{3}{2}(x+1)^2 (y+1)^3 + \frac{3}{8}x \in {R}_{3,6}.\\
\endaligned
$$
Then $Q=(x+1)(y+1)^2$ and $Z=Z^{(0)}=(x+1)(y+1)$. So the EN vertex of $N^0(Z)$ is $(1,1)$, which is in $\mathcal{L}_{(m,n)}$. Applying the decreasing automorphism $\varphi^{(0)}$ given by $x\mapsto x$ and $y\mapsto y-1$, we have 
$F^{(1)}=\varphi^{(0)}(F)=(x+1)^2 y^4 + (x+1)y$, 
$Z^{(1)}=\varphi^{(0)}(Z)=(x+1)y$.
Then $v_{F^{(0)}}=(2,4)$, $v_{F^{(1)}}=(1,1)$, 
 $v_{Z^{(0)}}=v_{Z^{(1)}}=(1,1)$,  
 $w_0=(1,0)$, $w_1=(3,-1)$,
 ($\star_0$) holds and ($\star_1$) fails. 
Then 
$$\deg_{w_1}( [ F^{(1)}, G^{(1)} ] )= \deg_{w_1}( F^{(1)} ) + \deg_{w_1}( G^{(1)} ) -\deg_{w_1}((1,1))= 2 + 3 - 2 =3.$$
Indeed we have $[F,G]= [ F^{(1)}, G^{(1)} ] =-\frac{3}{8}(x+1)$, whose $w_1$-degree is $3$. 
\end{example}

\begin{remark}
If the EN vertex of $N^0(Z)$ is in $$\{(x,y)\in \mathbb{R}^2 \, : \, 0\le y< \frac{a-1}{a}n\quad\text{ and }\quad\mathfrak{m}(x-\frac{a}{a+b}) +\frac{a}{a+b}\le y\le \frac{n}{m}x \},$$then the above induction process does not give an immediate contradiction. Here is an example (see Figure \ref{big_pic}). 
Let $(a,b,m,n)=(2,3,48,528)$. 
Here $(Q^{(1)})_+^{w_1}=x^2(xy^{12}-1)^{22}$ for $w_1=(12,-1)$, and 
 $\varphi^{(1)}:\mathbb{C}[x^{\pm 1/12},y] \longrightarrow \mathbb{C}[x^{\pm 1/12},y]$ is given by $x\mapsto x$ and $y \mapsto y + x^{-1/12}$. 
One of the edges of $N^0(Z^{(2)})$ joins $(1/3,0)$ and $(23/6,21)$, and one of the edges of $N^0(F^{(2)})$ joins $(1/3,0)$ and $(46/6,44)$. Both edges contain the point $(2/5,2/5)$, and they are normal to the direction $w_2=(6,-1)$. Then 
$$\aligned 
\deg_{w_2}( [ F^{(2)}, G^{(2)} ] )&= \deg_{w_2}( F^{(2)} ) + \deg_{w_2}( G^{(2)} ) -\deg_{w_2}((1,1))\\&=\deg_{w_2}( (2/5,2/5) ) + \deg_{w_2}( (3/5,3/5) ) -\deg_{w_2}((1,1))=0,\endaligned$$
which does not lead to an immediate contradiction.
\end{remark}

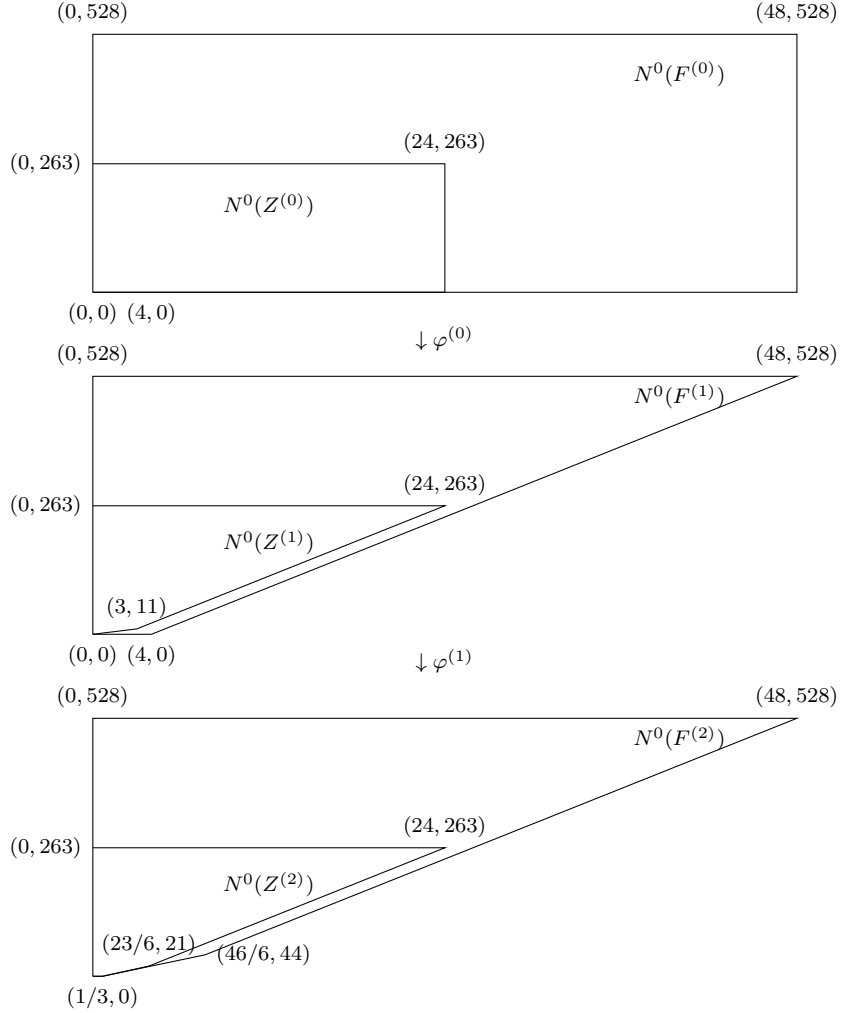
\begin{figure}[h]
\begin{tikzpicture}[scale=0.13]
\usetikzlibrary{patterns}
\draw (0,0)--(24*3,0)--(24*3, 264/10)--(0, 264/10)--(0,0);
\draw (0,0) node[anchor=north] {\tiny $(0,0)$};
\draw (2*3,0) node[anchor=north] {\tiny $(4,0)$};
\draw (24*3, 264/10) node[anchor=south] {\tiny $(48,528)$};
\draw (0, 264/10) node[anchor=south] {\tiny $(0,528)$};
\draw (0,263/20)--(24*3/2, 263/20)--(24*3/2,0)--(0,0);
\draw (0, 263/20) node[anchor=east] {\tiny $(0,263)$};
\draw (24*3/2, 263/20) node[anchor=south] {\tiny $(24,263)$};
\draw (24*3/4, 263/40) node[anchor=south] {\tiny $N^0(Z^{(0)})$};
\draw (20*3, 200/10) node[anchor=south] {\tiny $N^0(F^{(0)})$};
\begin{scope}[shift={(0,-35)}]
\usetikzlibrary{patterns}
\draw (24*3/2, 28) node[anchor=south] {\tiny $\downarrow\varphi^{(0)}$};
\draw (0,0)--(2*3,0)--(23/6*3,22/10)--(24*3, 264/10)--(0, 264/10)--(0,0);
\draw (0,0) node[anchor=north] {\tiny $(0,0)$};
\draw (2*3,0) node[anchor=north] {\tiny $(4,0)$};
\draw (24*3, 264/10) node[anchor=south] {\tiny $(48,528)$};
\draw (0, 264/10) node[anchor=south] {\tiny $(0,528)$};
\draw (0,263/20)--(24*3/2, 263/20)--(23/6*3/2,21/20)--(3*3/2,11/20)--(0,0);
\draw (0, 263/20) node[anchor=east] {\tiny $(0,263)$};
\draw (24*3/2, 263/20) node[anchor=south] {\tiny $(24,263)$};
\draw (3*3/2,11/20) node[anchor=south] {\tiny $(3,11)$};
\draw (24*3/4, 283/40) node[anchor=south] {\tiny $N^0(Z^{(1)})$};
\draw (20*3, 220/10) node[anchor=south] {\tiny $N^0(F^{(1)})$};
\end{scope}
\begin{scope}[shift={(0,-70)}]
\draw (24*3/2, 30) node[anchor=south] {\tiny $\downarrow\varphi^{(1)}$};
\draw (0,0)--(1/3*3,0)--(23/6*3,22/10)--(24*3, 264/10)--(0, 264/10)--(0,0);
\draw (1/3*3,0) node[anchor=north] {\tiny $(1/3,0)$};
\draw (24*3, 264/10) node[anchor=south] {\tiny $(48,528)$};
\draw (0, 264/10) node[anchor=south] {\tiny $(0,528)$};
\draw (0,263/20)--(24*3/2, 263/20)--(23/6*3/2,21/20)--(1/3*3,0/20)--(0,0);
\draw (0, 263/20) node[anchor=east] {\tiny $(0,263)$};
\draw (24*3/2, 263/20) node[anchor=south] {\tiny $(24,263)$};
\draw (23/6*3/2,21/20) node[anchor=south] {\tiny $(23/6,21)$};
\draw (23/6*3,22/10) node[anchor=west] {\tiny $(46/6,44)$};
\draw (24*3/4, 283/40) node[anchor=south] {\tiny $N^0(Z^{(2)})$};
\draw (20*3, 220/10) node[anchor=south] {\tiny $N^0(F^{(2)})$};
\end{scope}
\end{tikzpicture}
\caption{An example for $(a,b,m,n)=(2,3,48,528)$}
\label{big_pic}
\end{figure}

Similar to Proposition \ref{no12v2}, we have the following:
\begin{proposition}\label{no12v2_opposite}
Let $(a, b, m, n) \in \mathcal{Q}$ and assume $m\ge  n$.
Suppose that $\ff,\GG\in \mathbb{C}[x,y]$ satisfy the following:

\noindent\emph{(1)} $[\ff,\GG]  \in  \mathbb{C}$;

\noindent\emph{(2)} $\ff\in R_{m,n}$ and $\GG\in R_{bm/a,bn/a}$.

Let $Z$ be the innermost polynomial of $F$. Then the EN vertex of $N^0(Z)$ is  in $$\{(x,y)\in \mathbb{R}_{\ge0}^2 \, : \, y\ge \frac{n}{m}x\}.$$ 
\end{proposition}

\begin{proof}
The proof is similar to the one of Proposition \ref{no12v2}, so we only point out the difference.
The condition (d) in the definition of $\mathcal{T}$ becomes

\noindent{(d)} $v_z\in \{(x,y)\in \mathbb{R}^2 \, : \, 0\le y< \frac{n}{m}x\}$.

Aiming at contradiction, suppose that the EN vertex of $N^0(Z)$ is in 
$$\{(x,y)\in \mathbb{R}^2 \, : \, 0\le y< \frac{n}{m}x\}.$$ 

In Step 1, to check that $(F^{(j+1)}, G^{(j+1)}, Z^{(j+1)}, v_{Z^{(j+1)}}, v_{F^{(j+1)}})$ satisfies (d):  assume that 
$\deg_y(Z^{(j)})>0$ and construct $D_j$, $D_{j+1}$, $E$ as before.
We can still conclude that
$v_{F^{(j+1)}}$ and $E$ are strictly on the right side of the line $D_jv_{F^{(j)}}$,
and
$\text{deg}_y(v_{Z^{(j+1)}})\le\text{deg}_yE$, which implies the condition (d) that 
$v_{Z^{(j+1)}}$  is below the line $Ov_{F^{(j)}}$, and thus is in $\{(x,y)\in \mathbb{R}^2 \, : \, 0\le y< \frac{n}{m}x\}$.

In Step 2 Case 1:  $v(Z^{(j)})$ lies on ${\rm Line}_{w}(v_{F^{(j)}})$. Since $C$ is on the left side of ${\rm Line}_{w}(v_{F^{(j)}})$,  we have 
$$\deg_w( F^{(j)} ) > \deg_w(C).$$
The rest of the proof is almost identical.
\end{proof}

\begin{proposition}\label{no14v2}
Under the same assumption of Proposition \ref{no12v2}, the NE vertex of $N^0(Z)$ must be in $\{(x,y)\in \mathbb{R}_{\ge0}^2 \, : \, x\ge \frac{m}{n}y\}$.
\end{proposition}
\begin{proof}
Consider $\mathcal{U}^{(1,4)}$  instead of $\mathcal{U}^{(1,2)}$ and swap the roles of $x$ and $y$ in Proposition \ref{no12v2_opposite}.
\end{proof}

\begin{proof}[Proof of Thoerem~\ref{main_thm} ]
We claim that the EN vertex $P=(i,j)$ and the NE vertex $P'=(i',j')$ of $N^0(Z)$ coincide:
Indeed, if $P\neq P'$, then $i> i'$ and $j< j'$. Since both $P$ and $P'$ lie in the narrow region described in Remark \ref{remark:narrow region}, the point $P''=(i,j')$ also lies in that region, but then $P$ and $P''$ are distinct lattice points with the same $y$-coordinate, contradicting Remark \ref{remark:narrow region}.

The theorem then follows from Propositions~\ref{no12v2} and~\ref{no14v2}.
\end{proof}

\section{Conjecture~\ref{main_conj5}}
In this section we introduce Conjecture~\ref{main_conj5}, which implies the Jacobian conjecture.
\begin{definition}\label{def:T}
Let $T_{m,n,a}$ be the set of polynomials $f\in\mathcal{R}$ such that 

\noindent $\bullet$ $(m,n)\in N^0(f) \subseteq \Rect_{m,n}$;

\noindent $\bullet$ $f_n^{(0,1)}=x^m y^n$;

\noindent $\bullet$ $f_{n-1}^{(0,1)}=0$;

\noindent $\bullet$ and $z_s^{(0,1)}$ is a monomial, where $z$ is the inner polynomial of the pair $(f,a)$ and $s=\deg_y(z)$. Moreover $\text{supp}(z_s^{(0,1)})$ lies in $\mathfrak{R}$. 
\end{definition}

\begin{proposition}\label{prop:E}
Suppose that the Jacobian conjecture fails. Then there is a pair of polynomials $F,G\in\mathcal{R}$ such that

\noindent\emph{(1)}  $[F,G]\in \mathbb{C}\setminus\{0\}$;

\noindent\emph{(2)}  $F\in T_{m,n,a}$ and $G\in T_{bm/a,bn/a,b}$.

Equivalently, if there is no such pair, then the Jacobian conjecture is true.
\end{proposition}
\begin{proof}
It is well known \cite{A, ApOn, CN, H, L, M, MW, Moh, Na1, Nagata, NN1, NN2, Ok, Re, GGV} that if $F',G'\in\mathcal{R}$ satisfy $[F',G']\in \mathbb{C}\setminus\{0\}$ then there exists an automorphism $\psi$ and a lattice point $(m,n)$ such that $N^0(\psi(F'))\cap \{(x,y)\in \mathbb{R}^2 \, : \,  y=x-m+n \}=\{(m,n)\}$ and $N^0(\psi(F'))\cap \{(x,y)\in \mathbb{R}^2 \, : \,  x>m \}=\emptyset$. Then $(\psi(F'))_{n-1}^{(0,1)}$ is a monomial. More precisely, $(\psi(F'))_{n-1}^{(0,1)}=\alpha x^m y^{n-1}$ for some $\alpha\in\mathbb{C}$. Apply a new automorphism $\varphi$ that sends $y\mapsto y-\alpha/n$, and let $F=\varphi\psi(F')$, $G=\varphi\psi(G')$.
Then Theorem~\ref{main_thm} implies that $F$ and $G$ satisfy the conditions (1) and (2).
\end{proof}

\begin{definition}\label{def:Tbar}
We extend the definition of $T_{m,n,a}$ to  $\overline{T}_{m,n,a}$ . Fix a positive integer $\mathbf{d}$. Let $K_\mathbf{d}=\mathbb{C}[x^{\pm 1/\mathbf{d}}]$, and let $\overline{\mathcal{R}}=K_\mathbf{d}[y]$.

For any $m,n\in\mathbb{Z}_{>0}$,   let $\overline{T}_{m,n,a}$ be the set of elements $f$ in $\overline{\mathcal{R}}$ such that

\noindent $\bullet$ $(m,n)\in N^0(f) \subseteq \Rect_{m,n}$ (in particular, $f\in\mathbb{C}[x^{1/\mathbf{d}},y]$); 

\noindent $\bullet$ $f_n^{(0,1)}=x^m y^n$;

\noindent $\bullet$ $f_{n-1}^{(0,1)}=0$;

\noindent $\bullet$ and $z_s^{(0,1)}$ is a monomial, where $z$ is the inner polynomial of $(f,a)$ and $s=\deg_y(z)$. Moreover $\text{supp}(z_s^{(0,1)})$ lies in $\mathfrak{R}$.
\end{definition}

\begin{conjecturealpha}\label{main_conj5}
Let $F,G\in\overline{\mathcal{R}}$ be a pair of elements  such that

\noindent\emph{(1)}  $[F,G]\in \mathbb{C}\setminus\{0\}$;

\noindent\emph{(2)}  $F\in \overline{T}_{m,n,a}$ and $G\in \overline{T}_{bm/a,bn/a,b}$.

\noindent\emph{(3)}  $N^0(\ff)$ is similar to $N^0(\GG)$ with
the origin as center of similarity and with ratio $\deg(\ff) : \deg(\GG) = a : b$. 

\noindent Then $(\frac{a}{a+b},0)\notin N^0(F)$ and $(\frac{b}{a+b},0)\notin N^0(G)$.
\end{conjecturealpha}

\begin{corollary}[Corollary of Proposition \ref{prop:E}]
Conjecture \ref{main_conj5} implies  the Jacobian conjecture.
\end{corollary}
\begin{proof}
If the Jacobian conjecture is false, then there exist $F,G$ satisfying (1) and (2) of Proposition \ref{prop:E}. But then $(1,0)\in N^0(F)$, which contradicts the assertion $(\frac{a}{a+b},0)\notin N^0(F)$ of Conjecture \ref{main_conj5}. 
\end{proof}

\section{Proof of Conjecture \ref{main_conj5} for a special case}

In this section we prove Conjecture \ref{main_conj5} for the special case where $a=n$. The main purpose is to illustrate how Theorem \ref{main_thm} might help us eventually prove this conjecture, and why those $e_j$ for $j>0$ are not relevant in the argument.

We set 
$$\mathcal{C}=\{(\aaa,\bbb,n)\in \mathbb{Z}_{>0}^3\, : \  n/\aaa\in\mathbb{Z}\text{ but }\bbb/\aaa\not\in\mathbb{Z}\text{ and }\aaa/\bbb\not\in\mathbb{Z}\}.$$

Fix  $(\aaa,\bbb,n)\in \mathcal{C}$ and $\mathbf{d}\in\mathbb{Z}_{>0}$. Let $a=\aaa/\gcd(\aaa,\bbb)$ and $b=\bbb/\gcd(\aaa,\bbb)$. 
Let $s\in [0,\frac{a-1}{a}n-1]$ and $\mathcal{K}_\mathbf{d}=K_\mathbf{d}[\tilde{q}_0, \tilde{q}_1, ..., \tilde{q}_{n/a-2}, \tilde{z}_0, \tilde{z}_1,...,\tilde{z}_{s-1}]$. 
Define $$\aligned
\tilde{q}&=\tilde{q}_{n/a} y^{n/a} + \tilde{q}_{n/a -1} y^{n/a -1} + \cdots + \tilde{q}_{0}  \in \mathcal{K}_\mathbf{d}[y],\\
\tilde{z}&=\tilde{z}_{s} y^{s} + \tilde{z}_{s -1 } y^{s -1} + \cdots + \tilde{z}_{0}  \in \mathcal{K}_\mathbf{d}[y],
\endaligned$$
where 
both $\tilde{q}_{n/a}$ and $\tilde{z}_{s}$ are nonzero monomials in $K_{\mathbf{d}}$, and  $\tilde{q}_{n/a -1}=0$.
For any $(e_0,...,e_{\aaa+\bbb-1})\in \mathbb{C}^{\aaa+\bbb}$ with $e_0=1$,  and for 
each $i\in\mathbb{Z}_{\ge - b n/a}$, we define $\tilde{h}_{-i}\in \mathcal{K}_\mathbf{d}$ by
\begin{equation}\label{eq:e and h}
\sum_{j=0}^{\aaa+\bbb-1} e_j (\tilde{q}^{a} +   \tilde{z})^{(\bbb-j)/\aaa} =\sum_{i=- b n/ a}^{\infty}  \tilde{h}_{-i} y^{-i}\in \mathcal{K}_\mathbf{d}[y][[y^{-1}]],
\end{equation}
where the expansion of $(\tilde{q}^{a} +   \tilde{z})^{(\bbb-j)/\aaa}$ is given by inverting $\tilde{q}_{n/a}y^{n/a}$. 
The extended Magnus' formula motivates the above definition of $\tilde{h}_{i}$ and the definition of $h_{i}$ in \eqref{eq:h}. See the beginning of the proof of Conjecture \ref{main_conj5} for more details of the motivation.

\subsection{The case that $a=n(=\aaa)$} 
In this case, $b=\bar{b}$. 
Note that $0\le s \le \frac{a-1}{a}n-1=a-2$.
Write $\tilde{q}_1=p^{a-s}$ and $\tilde{z}_s=\zeta^{a-s}$ for some $p,\zeta \in K_{\mathbf{d}(a-s)}$. 
Then $\tilde{q}=p^{a-s}y$, $\tilde{z}=\zeta^{a-s}y^s+\tilde{z}_{s-1}y^{s-1}+\cdots+\tilde{z}_0$. 
Denote 
$$d_p :=\deg_xp=\frac{m}{(a-s)a},\quad d_\zeta:=\deg_x\zeta$$ 
Note that 
\begin{equation}\label{eq:dzeta}
\frac{sm}{(a-s)a} \le d_\zeta \le \frac{m}{a-s}-\frac{m-\frac{a}{a+b}}{a-\frac{a}{a+b}}
\end{equation}
which follows from the assumption of $\text{Supp}(z_s^{(0,1)})$ (see  Definition \ref{def:Tbar}). 
Define  
$$\tilde{w}_{i} := \frac{p^{(s-i)a}}{\zeta^{a-i}}\tilde{z}_{i} \in\mathcal{K}_{\mathbf{d}(a-s)}\; \textrm{for }i\in [0,s];\quad  
\tilde{w}_{i}:=0\; \textrm{ for  }i\notin [0,s]. 
$$
 In particular, $\tilde{w}_{s}=\frac{p^{0}}{\zeta^{a-s}}\tilde{z}_s=1$.

\begin{lemma}
For every nonnegative integer $i$, we have  
$$\frac{p^{ai-(a-s)b}\tilde{h}_{b-i}}{\zeta^i}
=\sum_{j=0}^{a+b-1} e_j \Big(\frac{p^{s}}{\zeta}\Big)^j 
 \sum_{an_0+(a-1)n_1+\cdots+(a-s)n_s=i-j}
 \binom{(b-j)/a}{n_0,\dots,n_s}\prod_{l=0}^{s-1} \tilde{w}_{l}^{n_l} 
$$ 
where $n_0,\dots,n_{s}\in\mathbb{Z}_{\ge0}$. (See \S\ref{section:Appendix B} for the definition of the multinomial coefficient $\binom{(b-j)/a}{n_0,\dots,n_s}$.)
\end{lemma}
\begin{proof}
The equation \eqref{eq:e and h} gives 
$$\aligned
\sum_{i=- b}^{\infty}  \tilde{h}_{-i} y^{-i}
&= \sum_{j=0}^{a+b-1} e_j( p^{(a-s)a}y^{a} +   \tilde{z})^{(b-j)/a}\\ 
&=\sum_{j=0}^{a+b-1} e_jp^{(a-s)(b-j)}y^{b-j} \Big( 1 +   \sum_{l=0}^s \tilde{z}_{s-l}\frac{y^{s-l-a}}{p^{(a-s)a}} \Big)^{(b-j)/a} \\ 
&=\sum_{j=0}^{a+b-1} e_jp^{(a-s)(b-j)}y^{b-j}  \Big( 1 +   \sum_{l=0}^s \tilde{w}_{{s}-l} \big( \frac{\zeta}{p^ay}\big)^{a-s+l} \Big)^{(b-j)/a} \\ 
\endaligned
$$
Then  
$$\aligned
\tilde{h}_{b-i}
&=\sum_{j=0}^{a+b-1} e_j p^{(a-s)(b-j)}
\Big[ 
 \Big( 1 +   \sum_{l} \tilde{w}_{a-s} \big( \frac{\zeta}{p^ay}\big)^{l} \Big)^{(b-j)/a}
\Big]_{y^{j-i}}\\
&=\sum_{j=0}^{a+b-1} e_j p^{(a-s)(b-j)}
 \sum_{an_0+(a-1)n_1+\cdots+(a-s)n_s=i-j}
 \binom{(b-j)/a}{n_0,\dots,n_s}\prod_l \Big(\tilde{w}_{a-l} \big( \frac{\zeta}{p^a}\big)^l \Big)^{n_{a-l}} \\
&=\sum_{j=0}^{a+b-1} e_j p^{(a-s)(b-j)-a(i-j)}\zeta^{i-j}
 \sum_{an_0+(a-1)n_1+\cdots+(a-s)n_s=i-j}
 \binom{(b-j)/a}{n_0,\dots,n_s}\prod_l \tilde{w}_{a-l} ^{n_{a-l}} \\
\endaligned
$$ 
Replacing $l$ by $a-l$ and dividing both sides by $p^{(a-s)b-ai}\zeta^i$. Note that $l$ runs from $0$ to $s-1$. Thus we obtain the expected identity. 
\end{proof}

\begin{lemma}\label{lem:h another}
 Let $\tilde{z}_{{s}}=0$, 
  $p=1$ and  define $\tilde{h}^*_{-i}\in  \mathbb{C}[\tilde{z}_1,...,\tilde{z}_{s-1}]$ by
\begin{equation}\label{eq:h*}
(y^{a} +  \tilde{z}_{s -1 } y^{s -1} + \cdots + \tilde{z}_{0} )^{b/a} =\sum_{i=- b}^{\infty}  \tilde{h}^*_{-i} y^{-i}\in \mathbb{C}[\tilde{z}_1,...,\tilde{z}_{s-1}][y][[y^{-1}]],
\end{equation}
Then
$$\tilde{h}^*_{b-i}
= \sum_{an_0+(a-1)n_1+\cdots+(a-s+1)n_{s-1}=i} 
 \binom{b/a}{n_0,\dots,n_{s-1}}\prod_{l=0}^{s-1} \tilde{z}_{l} ^{n_l}.  $$
 where $n_0,\dots,n_{s-1}\in\mathbb{Z}_{\ge0}$. 
\end{lemma}
\begin{proof}
The equation \eqref{eq:e and h} gives 
$$\aligned
\sum_{i=- b}^{\infty}  \tilde{h}^*_{-i} y^{-i}
&=( y^{a} +   \tilde{z})^{b/a} 
=y^{b} \Big( 1 +   y^{-a} \sum_{l=a-s+1}^a  y^{a-l}\tilde{z}_{a-l} \Big)^{b/a} 
=y^{b}  \Big( 1 +   \sum_{l=a-s+1}^a  y^{-l}\tilde{z}_{a-l} \Big)^{b/a} \\ 
\endaligned
$$
Then
$$\aligned
\tilde{h}^*_{b-i}
=\Big[  y^{b}  \Big( 1 +   \sum_{l=a-s+1}^a  y^{-l}\tilde{z}_{a-l} \Big)^{b/a}\Big]_{y^{b-i}}
= \sum_{an_0+(a-1)n_1+\cdots+(a-s+1)n_{s-1}=i} 
 \binom{b/a}{n_0,\dots,n_{s-1}}\prod_{l=0}^{s-1} \tilde{z}_{l} ^{n_l} \\
\endaligned
$$ 
\end{proof}

Define the following polynomials 
$\hat{h}_{b-i}\in K_{\mathbf{d}(a-s)}[x_0,\dots,x_{s-1}]$ for all $i\ge0$:
\begin{equation}\label{eq:hhat}
\hat{h}_{b-i}(x_0,\dots,x_{s-1})= 
\sum_{j=0}^{a+b-1}e_j \Big(\frac{p^s}{\zeta}\Big)^j
\sum_{an_0+\cdots+(a-s)n_s = i-j} 
 \binom{(b-j)/a}{n_0,\dots,n_s}\prod_{l=0}^{s-1} x_l^{n_l} 
\end{equation}

For a given $K_{\mathbf{d}(a-s)}$-algebra homomorphism $\phi:\mathcal{K}_{\mathbf{d}(a-s)}\longrightarrow K_{\mathbf{d}(a-s)}$, and for all $i\ge0$, define 
\begin{equation}\label{eq:h}
\aligned
h_{b-i}
&=\phi \bigg(\frac{p^{ai-(a-s)b}\tilde{h}_{b-i}}{\zeta^i}\bigg)
=\frac{p^{ai-(a-s)b}\phi(\tilde{h}_{b-i})}{\zeta^i}\\
&=\hat{h}_{b-i}\big(\phi(\tilde{w}_0),\dots,\phi(\tilde{w}_{s-1})\big)
=\phi\big(\hat{h}_{b-i}(\tilde{w}_0,\dots,\tilde{w}_{s-1})\big)
\in K_{\mathbf{d}(a-s)}
\endaligned
\end{equation}

Given $F\in \overline{T}_{m,n,a}$, we can write $F=q^a+z$, where $z=z_sy^s+z_{s-1}y^{s-1}+\cdots+z_0$, and $s=\deg_yz\le n-2$ by the definition of $T_{m,n,a}$. Note that $F$ determines a homomorphism $\phi:\mathcal{K}_{\mathbf{d}(a-s)}\to K_{\mathbf{d}(a-s)}$ by $\phi(\tilde{z}_i)=z_i$. Then $h_i$ are determined by \eqref{eq:h}.

\begin{proof}[Proof of Conjecture \ref{main_conj5} in  the case $n=a$] 
Assume that $F,G$ satisfy the conditions (1)--(3), and $(\frac{a}{a+b},0) \in N^0(F)$ (or equivalently, $(\frac{b}{a+b},0) \in N^0(G)$).
Note that $s\neq0$, otherwise \eqref{eq:dzeta} becomes $0\le d_\zeta \le \frac{m}{a}-\frac{m-\frac{a}{a+b}}{a-\frac{a}{a+b}}$, but the assumption $(\frac{a}{a+b},0)\in N^0(F)$ implies $d_\zeta\ge \frac{1}{a+b}$, and it is easy to verify that $\frac{1}{a+b}>\frac{m}{a}-\frac{m-\frac{a}{a+b}}{a-\frac{a}{a+b}}$, which leads to a contradiction.  

We first give the motivation of introducing $\tilde{h}_{i}$ in \eqref{eq:e and h} and $h_{i}$ in \eqref{eq:h}, and describe the idea of the proof. By the extended Magnus' formula (Proposition \ref{prop:Magnus_formula_extended}), the homogeneous components of $G$ are determined by a linear combination of rational powers of $F$. The polynomial $\tilde{h}_{i}$ is introduced such that $\phi(\tilde{h}_i)$ is the $\deg_y=i$ component of $G$ (for $i$ not too negative). The polynomial $h_{i}$ is obtained from $\phi(\tilde{h}_i)$ by multiplying by a rational power of $x$. Since $G$ has no negative components, $h_i=0$ for $i<0$ but not too negative. It turns out that we do not need the full strength of these equations, but only the coefficient $[h_i]_{x^{d_{h_i}}}=0$ at the so-called expected $x$-degree $d_{h_i}$. This gives us a system of equations $[\hat{h}_{i}]^{\bf w}_{b-i}=0$; see \eqref{eq:hi di-}. This system of equations coincides with the one obtained by another pair of carefully chosen polynomials $(F^*,G^*)$, from which we easily get the desired contradiction.

We extend the definition of $\deg_x(f)$ for $f\in K_{\mathbf{d}(a-s)}$ by defining $\deg_x(0)=-\infty$. 

We denote 
$[f]_i=[f]_{x^i}\cdot x^i$ to be the degree-$i$ term of $f$.

Let ${\bf w}=(a ,a-1,...,a-s+1)$. 
For $\hat{f}=\sum_\alpha c_\alpha \mathbf{x}^\alpha\in \mathbb{C}[x_0,\dots,x_{{s}-1}]$, $d\in\mathbb{Q}$, define $[\hat{f}]^{\bf w}_{d}$ to be the sum of terms of $\hat{f}$ with ${\bf w}$-weight $d$, that is, $[\hat{f}]^{\bf w}_{d}=\sum_{\alpha\cdot{\bf w}=d} c_\alpha \mathbf{x}^\alpha$.

Define 
$${\bf v}=(v_0,\dots,v_{{s}-1})= \big( \deg_x \phi(\tilde{w}_{0}), \deg_x \phi(\tilde{w}_{1}), \dots, \deg_x \phi(\tilde{w}_{{s}-1})\big)\in (\mathbb{Q}_{-\infty})^{s}.$$  
There exists a unique $k\in\mathbb{Q}$ such that
${\bf v}^+ = (v^+_0,\dots,v^+_{s-1}) := k{\bf w}$ satisfies the condition that ${\bf v}^* = {\bf v}^+ -{\bf v}\in (\mathbb{Q}_{\ge0}\cup\infty)^{s}$  has at least one zero coordinate. 
In other words, $k$ is the minimum choice such that ${\bf v}^*$ has no negative coordinate. 
Note that 
$v_i^+ = k(a-i)\ge v_i$ for $0\le i\le s-1$, and the equality holds for at least one $i$. 

Next we give a lower bound for $k$. Since 
$(\frac{b}{a+b},0)\in N^0(G)$, we have $\deg_x\phi(\tilde{h}_0)=\deg_x G_0\ge \frac{b}{a+b}$. Then 
$
kb \ge \deg_xh_0=\deg_x\frac{p^{sb}}{\zeta^b} \tilde{h}_0 \ge sbd_p-bd_\zeta+\frac{b}{a+b}
$. By \eqref{eq:dzeta}, we have 
\begin{equation}\label{eq:kge}
k\ge sd_p-d_\zeta+\frac{1}{a+b}
\ge \frac{sm}{(a-s)a}
-\Big(\frac{m}{a-s}-\frac{m-\frac{a}{a+b}}{a-\frac{a}{a+b}}\Big)
+\frac{1}{a+b}
=\frac{m-\frac{a}{a+b}}{a(a+b-1)}
>0.
\end{equation}

Let $h_i$ be as defined in \eqref{eq:h}. 
Define the expected $x$-degree of $h_i$ to be 
$$d_{h_i}:=k(b-i) \in\mathbb{Q}.$$ 
Note that $\deg_x h_i \le d_{h_i}$. Indeed, if $h_i=0$, then $\deg_x h_i=-\infty<d_{h_i}$; if $h_i\neq0$, then  
\begin{equation}\label{eq:deg hi>=di}
\aligned
\deg_x h_i 
&
\le \max_j \max_{an_0+\cdots+(a-s)n_s=b-i-j}
\deg_x\phi \big( (\frac{p^s}{\zeta})^j\prod_l \tilde{w}_{l}^{n_l} \big)\\
&\le \max_j \max_{ {\bf n}\cdot{\bf w} + (a-s)n_s =b-i-j} 
(sd_p-d_\zeta)j + {\bf n}\cdot {\bf v}^+\\
&= \max_j \max_{ {\bf n}\cdot{\bf w} + (a-s)n_s =b-i-j} 
( sd_p-d_\zeta)j + k(b-i-j-(a-s)n_s )\\
&\le \max_j \max_{ {\bf n}\cdot{\bf w} + (a-s)n_s =b-i-j} 
(sd_p-d_\zeta-k)j + k(b-i)\\
&\le k(b-i) = d_{h_i}\\
\endaligned
\end{equation}
where ${\bf n}=(n_0,\dots,n_{s-1})$, and the last ``$\le$'' is because by \eqref{eq:dzeta},
$$sd_p-d_\zeta-k \le s\frac{m}{(a-s)a}-\frac{sm}{(a-s)a}-k=-k < 0.$$ 
Note that for the last  ``$\le$''  to ``$=$'', we must have $j=0$. For the second last ``$\le$'' to be ``$=$'', we must have $n_s=0$.  So if both ``$\le$'' are ``$=$'', then ${\bf n}\cdot{\bf w}=b-i$.

We verify that the condition for Lemma \ref{lem:recurrence} (b) holds. Indeed, $[h_i]_{d_{h_i}} = 0$ holds for $-1\ge i\ge -a+2$ because  $h_i=0$. To show that 
$[h_{-a+1}]_{d_{h_{-a+1}}}=0$,  note that
$$
\deg_x h_{-a+1} 
= 
\deg_x \phi\Big(\frac{p^{a(a+b-1)-(a-s)b}\tilde{h}_{-a+1}}{\zeta^{a+b-1}}\Big)
= 
(a(a+b-1)-(a-s)b) d_p - (a+b-1) d_\zeta + (1-m).
$$
Thus by \eqref{eq:kge} and that $d_p=\frac{m}{(a-s)a}$, we have
$
d_{h_{-a+1}} - \deg_x h_{-a+1} 
\ge
\Big(sd_p -d_\zeta +\frac{1}{b}\Big)(a+b-1)- (a(a+b-1)-(a-s)b) d_p + (a+b-1) d_\zeta - (1-m)
= \frac{m}{a}+\frac{a-1}{b}>0
$. 
Thus $\deg_x h_{-a+1}<d_{h_{-a+1}}$, 
$[h_{-a+1}]_{d_{-a+1}}=0$. 
By Lemma \ref{lem:recurrence} (b) we conclude that 
\begin{equation}\label{eq:hi di-=0}
\textrm{$[h_i]_{d_{h_i}}=0$ holds for all $i<0$.} 
\end{equation}

Next, we claim that  
\begin{equation}\label{eq:hi di-}  
[h_i]_{d_i} 
= \sum_{{\bf n}\cdot{\bf w}=b-i} 
 \binom{b/a}{n_0,\dots,n_{s-1}}\prod_{j=0}^{s-1} [\phi(\tilde{w}_{j})]_{v^+_{j}}^{n_j}
= [\hat{h}_{i}]^{\bf w}_{b-i}\Big( [\phi(\tilde{w}_{0})]_{v^+_0}, \dots, [\phi(\tilde{w}_{{s}-1})]_{v^+_{{s}-1}} \Big) 
\end{equation}
Indeed, the first equality follows from \eqref{eq:deg hi>=di}. 
To show the second equality, note that  
\begin{equation}\label{eq:hat h_i}
\aligned
[\hat{h}_{i}]^{\bf w}_{b-i}
&= \sum_{j=0}^{a+b-1}e_j \Big(\frac{p^{s}}{\zeta}\Big)^j
\sum_{\stackrel{an_0+\cdots+(a-s)n_{s}=b-i-j}{{\bf n}\cdot{\bf w}=b-i}} 
 \binom{(b-j)/a}{n_0,\dots,n_{s}}\prod_{l=0}^{s-1} x_l^{n_l} \\
&= \sum_{{\bf n}\cdot{\bf w}=b-i} 
 \binom{b/a}{n_0,\dots,n_{s-1}}\prod_{l=0}^{s-1} x_l^{n_l}\\
 \endaligned  
\end{equation}
where we use the fact that if ${\bf n}\cdot{\bf w}=b-i$ and $an_0+\cdots+(a-s)n_s=b-i-j$, then $n_s=0$ and $j=0$. 
Now substituting $(x_0,\dots,x_{s-1})$ by 
$\big( [\phi(\tilde{w}_{0})]_{v^+_0}, \dots, [\phi(\tilde{w}_{{s}-1})]_{v^+_{{s}-1}} \big)$, we have shown 
 the second ``$=$'' of  \eqref{eq:hi di-}. 

Define
$\alpha=(\alpha_0,\dots,\alpha_{{s}-1}):=( [\phi(\tilde{w}_{0})]_{v^+_0}, [\phi(\tilde{w}_{1})]_{v^+_1}, \dots, [\phi(\tilde{w}_{{s}-1})]_{v^+_{{s}-1}})\in K_{\mathbf{d}(a-s)}^{{s}}
$.  
It follows from \eqref{eq:hi di-=0} and \eqref{eq:hi di-} that $[\hat{h}_{i}]^{\bf w}_{b-i}(\alpha)=0$ for all $i<0$. 

Define $F^*= \sum_{i=0}^a F^*_iy^i := y^a+\alpha_{{s}-1} y^{{s}-1}+\cdots+\alpha_1y+\alpha_0\in K_{\mathbf{d}(a-s)}[y]$, 
and then expand 
$(F^*)^{b/a}=\sum_{i=-b}^\infty g_{-i} y^{-i}\in K_{\mathbf{d}(a-s)}[y][[y^{-1}]]$ (compare with \eqref{eq:e and h}). 
Define $G^*=\sum_{i=-b}^0 g_{-i} y^{-i}\in K_{\mathbf{d}(a-s)}[y]$. Let $\phi^*:\mathcal{K}_{\mathbf{d}(a-s)}\to K_{\mathbf{d}(a-s)}$ be the $K_{\mathbf{d}(a-s)}$-homomorphism  sending $\tilde{z}_i$ to $\alpha_i$. 
By Lemma \ref{lem:h another} and the equality \eqref{eq:hat h_i}, we have 
$g_{i}=\phi^*(\tilde{h}^*_{i})  = [\hat{h}_i]^{\bf w}_{b-i}=0$ for $i<0$.
So $(F^*)^{b/a}=\sum_{i=0}^b g_{i} y^{i}=G^*$.
Since $F^*,G^*$ are in $K_{\mathbf{d}(a-s)}[y]$ which is an integrally closed domain,  and $a, b$ are coprime, we conclude that $(F^*)^{1/a}\in K_{\mathbf{d}(a-s)}[y]$, thus can be written as 
 $(F^*)^{1/a}=y+r$ with $r\in K_{\mathbf{d}(a-s)}$. 
 Then $F^*=(y+r)^a=y^a+ary^{a-1}+$(terms with $y$-degree $\le a-2$). But $F^*$ does not have a nonzero term of $y$-degree $a-1$, which implies 
 $ar=0$, 
 $r=0$, 
 $F^*=y^a$, 
 $\alpha=(0,\dots,0)$, 
a contradiction. 
 \end{proof}
 
The following lemma is used in the above proof. 
\begin{lemma}\label{lem:recurrence}
{\rm(a)} Assume $F=\sum_{i=0}^aF_i y^i$, $B=\sum_{i=-\infty}^b B_i y^i$ where $F_i, B_i\in K_\mathbf{d}$, and $F^{b/a}=B$. 
Then
$
\sum_{i=0}^a  \big( \frac{a+b}{a}i - j \big) F_i  B_{j-i}
=0
$. 

{\rm(b)} Assume $F$ is as in Conjecture \ref{main_conj5}, it induces $h_i$ as defined in \eqref{eq:h}. 
If the equality $[h_i]_{d_{h_i}}=0$ holds for $-1\ge i\ge -a+1$, then it holds for all $i<0$. 
\end{lemma}
\begin{proof}
(a) 
From $F^{b/a}=B$, we get 
$
\frac{b}{a}B \frac{\partial F}{\partial y}
=
F \frac{\partial B}{\partial y}
$, and 
$
\frac{b}{a} (\sum_{i=-\infty}^b B_i y^i) (\sum_{i=0}^aF_i iy^{i-1})
=
 (\sum_{i=0}^aF_i y^i) (\sum_{i=-\infty}^b B_i iy^{i-1}) 
$. 
Take the coefficients of $y^{j-1}$, we get 
$
\frac{b}{a} \sum_{i=0}^a i F_i B_{j-i}
=
 \sum_{i=0}^a (j-i) F_i  B_{j-i}$, 
 thus $\sum_{i=0}^a  \big( \frac{a+b}{a}i - j \big) F_i B_{j-i}
=0$. 

(b) 
For $0\le i\le a-3$, let $B_i=\phi(\tilde{h}_i)$ and define 
$d_{B_i}:= (b-i)(d_\zeta+k) - (a(b-i)-(a-s)b)d_p$, 
$d_{F_i} := (a-i)(d_\zeta+k) -(s-i)ad_p$. 
Note that 
$
\deg_x (F_i) 
= \deg_x \phi(\tilde{z}_i) 
= \deg_x (\frac{\zeta^{a-i}}{p^{(s-i)a}}\tilde{w}_i) 
=(a-i)d_\zeta -(s-i)ad_p+v_i
\le (a-i)d_\zeta -(s-i)ad_p+k(a-i) = d_{F_i}$.
Moreover, if we assume $v_a=v_s=0$ and $v_{i}=-\infty$ for $s<i<a$, then the above inequality also holds for $s\le i\le a$. 
Also note that
$\deg_x(B_i)
= \deg_x \phi(\tilde{h}_i) 
= \deg_x \phi( \frac{ \zeta^{b-i}}{p^{a(b-i)-(a-s)b} } h_i )
\le (b-i)d_\zeta - (a(b-i)-(a-s)b)d_p + d_{h_i}
=d_{B_i} 
$ 
and that
$
[h_i]_{x^{d_{h_i}}}=[B_i]_{x^{d_{B_i}}}
$. 

Define the expected degree of $F_i B_{j-i} = F_i\phi(\tilde{h}_{j-i})$ to be 
$
d_j:= d_{F_i}+d_{B_{j-i}} =  (a-i)(d_\zeta+k) -(s-i)ad_p +
(b-j+i)(d_\zeta+k) - (a(b-j+i)-(a-s)b)d_p
=  (a+b-j)(d_\zeta+k) + (aj-as-bs)d_p$ 
which is independent of $i$. Then by (a), 
\begin{equation}\label{eq:FB recurrence}
\aligned
\sum_{i=0}^a  \big( \frac{a+b}{a}i - j \big) [F_i]_{d_{F_i}}  
[h_{j-i}]_{d_{h_{j-i}}}
&=\sum_{i=0}^a  \big( \frac{a+b}{a}i - j \big) [F_i]_{d_{F_i}}  
[B_{j-i}]_{d_{B_{j-i}}} \\
&= \Big[\sum_{i=0}^a  \big( \frac{a+b}{a}i - j \big) F_i  B_{j-i}\Big]_{d_j}
=0.\\
\endaligned
\end{equation}
For any $j<0$, the above coefficients $\frac{b}{a}i +i - j\neq0$. Also note that $[F_a]_{d_{F_a}}=x^m\neq0$. 
So \eqref{eq:FB recurrence} gives a recurrence relation on $[h_i]_{d_{h_i}}$ of order $a$.  Moreover, for $j=0$, \eqref{eq:FB recurrence} does not have the term corresponding to $i=0$ term, thus becomes 
$$ 
\big( \frac{a+b}{a} \big) [F_1]_{d_{F_1}}  
[h_{-1}]_{d_{h_{-1}}}
+
\cdots
+
\big( \frac{a+b}{a}a \big) [F_a]_{d_{F_a}} 
[h_{-a}]_{d_{h_{-a}}}=0.$$
So ``$[h_i]_{d_{h_i}}=0$ holds for $-1\ge i\ge -a+1$'' implies that it also holds for $i=-a$, thus it holds for all $i<0$ because the recurrence relation \eqref{eq:FB recurrence} has order $a$.
\end{proof}

\section{Appendix A: uniqueness of the principal polynomial}\label{appendixA}

Lemma~\ref{unique_principal} below will show that $\mathcal{W}(f)$ contains a unique principal polynomial up to roots of unity. Before that, we want to first prove Lemma \ref{maximal_factorization} using algebraic geometry.

We recall some simple facts from algebraic geometry. A polynomial $f\in\mathbb{C}[x,y]$ determines a homomorphism $f:\mathbb{C}[z]\to \mathbb{C}[x,y]$ sending $z$ to $f$, thus induces a morphism $\phi_f: \mathbb{A}^2\to \mathbb{A}^1$ (which is the polynomial function determined by $f$). Conversely, every morphism $\mathbb{A}^2\to \mathbb{A}^1$ is determined by a unique polynomial. 
The following are equivalent: ``$f$ is not a constant function'' $\Leftrightarrow$ ``$\phi_f$ is a non-constant morphism'' $\Leftrightarrow$  ``$\phi_f$ is surjective''. Similarly, $\g\in \mathbb{T}$ induces a proper and finite morphism $\phi_\g:\mathbb{A}^1\to\mathbb{A}^1$ with $\deg \phi_\g=\deg \g$, and $W\in \mathbb{C}[x,y]$ induces a morphism $\phi_W:\mathbb{A}^2\to \mathbb{A}^1$. The condition $f=\g(W)$ translates to  $\phi_f=\phi_\g\circ\phi_W$, which we call the \emph{factorization}  determined by $W$.

In the following lemma, for convenience of notation, for $f\in\mathbb{C}[x,y]$ we write $\phi_f$ simply as $f$ when no confusion should occur. By ``degree of $\phi_f$'', we mean the the $(1,1)$-degree of $f$. For $\g\in\mathbb{C}[z]$, we write $\phi_\g$ as $\g$; by ``degree of $\phi_\g$'' we mean the degree of $\g$. 

Given  a non-constant morphism  $f: \mathbb{A}^2\to\mathbb{A}^1$, we consider factorizations of $f$ of the form $f=\g\circ j$ with morphisms $j:\mathbb{A}^2\to\mathbb{A}^1$ and $\g:\mathbb{A}^1\to\mathbb{A}^1$. 
We say that two factorizations $f=\g_1\circ j_1$ and $f=\g_2\circ j_2$ are equivalent if there is an isomorphism $i:\mathbb{A}^1\to\mathbb{A}^1$ such that $\g_2=\g_1\circ i$ and $j_1=i\circ j_2$. We call $\deg(\g)$ the \emph{depth} of the factorization $f=\g\circ j$. We call $f=\g\circ j$ a deepest factorization of $f$ if there exists no other factorization of $f$ with a larger depth.

\begin{lemma}\label{maximal_factorization} Given a non-constant morphism $f: \mathbb{A}^2\to\mathbb{A}^1$, there is, up to isomorphism,  a unique universal factorization $f=\g_0\circ j_0$ in the sense that: if there is another factorization $f=\g\circ j$, then there is a morphism $i: \mathbb{A}^1\to \mathbb{A}^1$ such that $\g_0=\g\circ i$ and $j=i\circ j_0$, that is, the following diagram commutes. (Equivalently, $f=\g_0\circ j_0$ is a deepest factorization of $f$.)
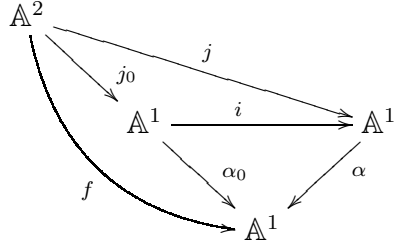
$$\xymatrix{
\mathbb{A}^2\ar[rd]^(.7){j_0}\ar[rrrd]^{j}\ar@/_2pc/[ddrr]_f&\\
&\mathbb{A}^1\ar[rd]^(.6){\g_0}\ar[rr]^(.4){i}&&\mathbb{A}^1\ar[ld]^(.3){\g}\\ 
&&\mathbb{A}^1
}$$
\end{lemma}
\begin{proof}
Given any two factorizations $f=\g_1\circ j_1=\g_2\circ j_2$, we can construct a new factorization $f=\g_3\circ j_3$ which is either equivalent to one of the two factorizations, or deeper than them. Note that since $f$ is non-constant, $\g_1,\g_2,j_1,j_2$ must also be non-constant. 
$$
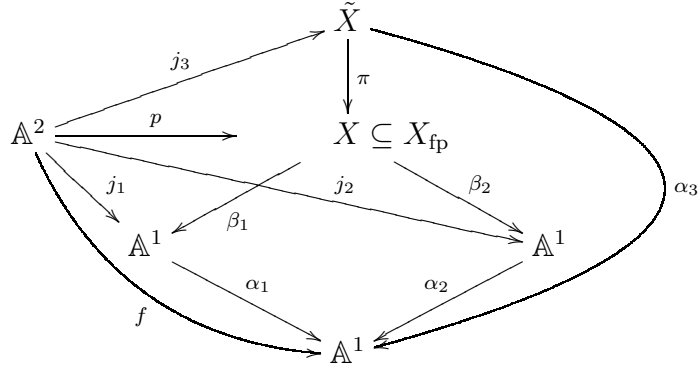
\xymatrix{
&&\tilde{X}\ar[d]^\pi\ar@/_-10pc/[ddd]^{\g_3}\\
\mathbb{A}^2\ar[rd]^(.6){j_1}\ar[rrrd]^(.6){j_2}\ar@/_2pc/[ddrr]_f\ar[rr]^(.4)p\ar[rru]^{j_3}&&\hspace{33pt}X\subseteq X_{\rm fp} \ar[ld]^(.6){\beta_1}\ar[rd]^(.6){\beta_2}\\
&\mathbb{A}^1\ar[rd]^{\g_1}&&\mathbb{A}^1\ar[ld]_{\g_2}\\ 
&&\mathbb{A}^1
}$$

Let $X_{\rm fp}$ be the fiber product of $\g_1$ and $\g_2$. Then $f$ factors through $X_{\rm fp}$, and its image in $X_{\rm fp}$ must be irreducible. Let $X$ be the irreducible component of $X_{\rm fp}$ that contains the image of $f$. Then $f$ factors through $X$ in the sense that there is a morphism $p:\mathbb{A}^2\to X$ such that $f=\g_1\circ \beta_1\circ p=\g_2\circ \beta_2\circ p$. Let $\pi:\tilde{X}\to X$ be the normalization of $X$, then $p$ factors through $\tilde{X}$ in the sense that there is a morphism $j_3:\mathbb{A}^2\to X$ such that $p=\pi\circ j_3$. Define $\g_3=\g_1\circ \beta_1\circ \pi (=\g_2\circ \beta_2\circ \pi) : \tilde{X}\to \mathbb{A}^1$. Then $f=\g_3\circ j_3$.
Note that the curve $X$ is rational, otherwise each line in $\mathbb{A}^2$ must map to a point in $X$ because there is no non-constant rational map from a rational curve to a nonrational curve
\footnote{Here is a short sketch of proof: assume $f: C_1\to C_2$ is a non-constant rational map, where $C_1$ is a rational curve and $C_2$ is a nonrational curve. Without loss of generality we may assume that $C_1$ and $C_2$ are nonsingular projective curves. So ${\rm genus}(C_1)=0$,  ${\rm genus}(C_2)\ge1$.  
By \cite[II.2.1]{Silverman}, a rational map from a smooth projective curve to a projective curve is always a morphism. 
Thus $f$ is a morphism. Applying Riemann--Hurwitz formula to $f$, we get a contradiction $-2=2{\rm genus}(\mathbb{P}^1)-2=(\deg f)(2{\rm genus}(C_2)-2)+\sum (e_P-1)\ge 0$ where $e_P\ge1$ are ramification indices.
 }.
Since $X$ is rational and affine (but possibly singular), $\tilde{X}$ is a nonsingular affine rational curve, thus $\tilde{X}\cong\mathbb{A}^1$. Moreover since $\g_1$, $\g_1'$, and $\pi$ are all proper and non-constant, the morphism $\g_3$ is proper and non-constant. 
Note that $\deg \g_3\ge \deg \g_1$ (resp.~$\deg \g_2$), and equalities hold when the factorization $f=\g_3\circ j_3$ is equivalent to $f=\g_1\circ j_1$ (resp.~$f=\g_2\circ j_2$). 

Consider the set of all equivalence classes of the factorizations of $f$. Note that the depth of any factorization is no larger than the $(1,1)$-degree of the polynomial $f$, so there exists a deepest factorization $f=\g_0\circ j_0$ (that is, $\deg(\g_0)$ is maximal). By the above argument, we see that $f=\g_0\circ j_0$ is universal. The uniqueness of the factorization follows from the universal property.  
This completes the proof. 
\end{proof}

\begin{lemma}\label{unique_principal}
Let $f$ be a non-constant polynomial in $\mathcal{R}$. Then $\mathcal{W}(f)$ contains a unique principal polynomial up to roots of unity.
\end{lemma}
\begin{proof}
As explained before Lemma~\ref{maximal_factorization}, each $W\in \mathcal{W}(f)$ determines a factorization of $f$. Two polynomials $W_1,W_2\in \mathcal{W}(f)$ determine equivalent factorizations if and only if the corresponding $\g_1$ and $\g_2$ satisfy $\phi_{\g_1}=\phi_{\g_2}\circ i$ for an isomorphism $i:\mathbb{A}^1\to\mathbb{A}^1$. Then $i(z)=az+b$ for some $a,b\in \mathbb{C}$ with $a\neq0$. Since $\g_2\in\mathbb{T}$, we can write $\alpha_2=z^k+e_{k-2}z^{k-2}+\cdots+e_0$, and $\alpha_1=(az+b)^k+e_{k-2}(az+b)^{k-2}+\cdots+e_0=a^kz^k+ka^{k-1}bz^{k-1}+\cdots$. Then $\g_1\in\mathbb{T}$ implies $a^k=1, ka^{k-1}b=0$. It follows that $a$ must be a $(\deg \g_1)$-th root of unity, and the constant $b=0$. 

By Lemma~\ref{maximal_factorization}, $W$ is a principal polynomial if and only if it corresponds to the  deepest factorization $f=\g_0\circ j_0$ which must exist and is unique. By the previous paragraph, the corresponding polynomial $W\in \mathcal{W}(f)$ is uniquely determined up to a ($\deg \g_0$)-th root of unity. 
\end{proof}

\section{Appendix B: the extended Magnus formula}\label{section:Appendix B}

First, recall the definition of the multinomial coefficients. For $a\in\mathbb{R}$, $k\in\mathbb{Z}_{>0}$ and $m_1,\dots,m_k\in \mathbb{Z}_{\ge0}$, denoting the multinomial coefficient

\noindent $\binom{a}{m_1,m_2,\dots,m_k}=\binom{a}{m_1}\binom{a-m_1}{m_2}\cdots \binom{a-m_1-\cdots-m_{k-1}}{m_k}=\frac{a(a-1)(a-2)\cdots(a-m_1-\cdots-m_k+1)}{m_1!m_2!\cdots m_k!}$.

In \cite[Theorem 1]{Magnus1}, Magnus produced a formula which inspired much of the work for this paper. 
In \cite{HLLN}, we have proved the following theorem. 

\begin{theorem}\label{Magnus_thm}
Suppose $[F, G] \in  \mathbb{C}$. For any direction $w=(u,v)$, let $d=w\text{-}\deg(F_+)$ and $e=w\text{-}\deg(G_+)$. Write the $w$-homogeneous degree decompositions $F=\sum_{i\leq d} F_i$ and $G=\sum_{i\leq e} G_i$. Then 
there exists a unique \footnote{Note that there is some ambiguity of the notation $F_d^{1/r}$, since it is unique up to an $r$-th root of unity. We fix a choice of $F_d^{1/r}$. Then the fractional power $F_d^{c/r}:=(F_d^{1/r})^c$ is nonambiguous for any integer $c$.} sequence of constants $c_0,c_1,...,c_{d+e-u-v-1}\in \mathbb{C}$ such that $c_0\neq 0$ and 
\begin{equation}\label{Magnus_formula}
G_{e-\mu} = \sum_{\gamma=0}^{\mu} c_\gamma \sum {(e-\gamma)/d \choose  \nu_{\gamma,0},\, \nu_{\gamma,1},\, \dots,\, \nu_{\gamma,d-1}} F_d^{(e-\gamma)/d - \sum_{\alpha\le d-1} \nu_{\gamma,\alpha}} \prod_{\alpha\le d-1} F_\alpha^{\nu_{\gamma,\alpha}}
\end{equation}
for every integer $\mu\in\{0,1,...,d+e-u-v-1\}$, where the inner sum is to run over all combinations of non-negative integers $\nu_{\gamma,\alpha}$ satisfying 
$\sum_{\alpha\le d-1} (d-\alpha)\nu_{\gamma,\alpha}=\mu-\gamma$.  Furthermore, $c_\gamma=0$ if $r(e-\gamma)/d\notin\mathbb{Z}$, where $r\in\mathbb{Z}_{>0}$ is the largest integer such that $F_d^{1/r}\in \mathbb{C}[x,y]$. 
\end{theorem}

We start by reinterpreting  the equality \eqref{Magnus_formula}. 
By the generalized multinomial theorem,  
for $x_1,\dots,x_n\in \mathcal{R}$, we have the following expansion in the ring $\mathcal{R}[[t]]$:
$$(1+x_1t+\cdots+x_{n}t^n)^A=\sum_{v_1,\dots,v_n\in\mathbb{Z}_{\ge0}}{A\choose v_1,v_2,\dots,v_n}x_1^{v_1}\cdots x_n^{v_n}t^{v_1+2v_2+\cdots+nv_n}.$$
In general, for $x_0,\dots,x_n\in \mathcal{R}$ and for $A=a/b$ where $a\in\mathbb{Z}$, $b\in\mathbb{Z}_{>0}$,   
we have the following identity in the ring $\mathcal{R}[x_0^{\pm 1/b}][[t]]$ (where we fix a choice of $x_0^{1/b}$): 
\begin{equation}\label{eq:xA}
(x_0+x_1t+\cdots+x_{n}t^n)^A=\sum_{v_1,\dots,v_n\in\mathbb{Z}_{\ge0}}
{A\choose v_1,v_2,\dots,v_n}
x_0^{A-\sum_{i=1}^n v_i}x_1^{v_1}\cdots x_n^{v_n}t^{v_1+\cdots+nv_n}.
\end{equation}

\begin{lemma}\label{reinterpret Magnus formula}
The equality \eqref{Magnus_formula} can be rewritten as the following equality in $\mathcal{R}[F_d^{\pm 1/d}]$:
\begin{equation}\label{eq:gemu_variation}
\aligned
G_{e-\mu} 
&= \sum_{\gamma=0}^{\mu} c_\gamma 
\bigg[\bigg(F_d+F_{d-1}t+F_{d-2}t^2+\cdots\bigg)^{(e-\gamma)/d}\bigg]_{t^{\mu-\gamma}}\\
\endaligned
\end{equation}
\end{lemma}

In \cite{HLLN}, we proved the following statement, which is equivalent to Theorem \ref{Magnus_thm}.
\begin{theorem}\label{prop:Magnus_formula_equivalence}
Suppose $[F, G] \in  \mathbb{C}$. For any direction $w=(u,v)$, let $d=w\text{-}\deg(F_+)$ and $e=w\text{-}\deg(G_+)$.  Assume $d>0$. Write the $w$-homogeneous  decompositions $F=\sum_{i\leq d} F_i$ and $G=\sum_{i\leq e} G_i$. Define 
$$\widetilde{F}=\ff_d+\ff_{d-1}t+\cdots \quad\text{and}\quad \widetilde{G}=G_e+G_{e-1}t+\cdots.$$
Let $r\in\mathbb{Z}_{>0}$ be the largest integer such that $F_d^{1/r}\in \mathbb{C}[x,y]$. 
Then there exists a unique sequence of constants $c_0,c_1,\dots,c_{d+e-u-v-1}\in \mathbb{C}$ such that $c_0\neq 0$ and 
\begin{equation}\label{Magnus_formula_equivalence}
G_{e-\mu} = \sum_{\gamma=0}^{\mu} c_\gamma [\widetilde{F}^{\frac{e-\gamma}{d}}]_{t^{\mu-\gamma}}
\end{equation}
for every integer $\mu\in\{0,1,...,d+e-u-v-1\}$. 
Moreover, $c_\gamma=0$ if $r(e-\gamma)/d\notin\mathbb{Z}$. 
\end{theorem}

When the direction is  $w=(0,1)$, we strengthen the theorem to an ``extended'' Magnus' formula by including an assertion for $\mu=d+e-u-v=d+e-1$.

\begin{theorem}\label{prop:Magnus_formula_extended}
Suppose $[F, G] \in  \mathbb{C}$. Let $w=(0,1)$, and let $d=w\text{-}\deg(F_+)=n$ and $e=w\text{-}\deg(G_+)$.  Assume $d>0$. Write the $w$-homogeneous  decompositions $F=\sum_{i\leq d} F_i$ and $G=\sum_{i\leq e} G_i$ , and assume $F_d=x^my^n$. Define 
$$\widetilde{F}=\ff_d+\ff_{d-1}t+\cdots \quad\text{and}\quad \widetilde{G}=G_e+G_{e-1}t+\cdots.$$
Let $r=\gcd(m,n)$. 
Then there exists a unique sequence of constants $c_0,c_1,\dots,c_{d+e-2}\in \mathbb{C}$ such that $c_0\neq 0$,  
\eqref{Magnus_formula_equivalence}
for every integer $\mu\in\{0,1,...,d+e-2\}$, and 
\begin{equation}\label{Magnus_formula_extended}
\frac{\lambda}{x^{m-1}y^{n-1}} = \sum_{\gamma=0}^{d+e-1} c_\gamma [\widetilde{F}^{\frac{e-\gamma}{d}}]_{t^{d+e-1-\gamma}}
\end{equation} 
for some constant $\lambda\in\mathbb{C}$. 
Moreover, $c_\gamma=0$ if $r(e-\gamma)/d\notin\mathbb{Z}$. 
\end{theorem}

Recall a useful lemma, which appears in 
\cite[Propositions 1,2]{NaBa}, \cite[Lemma 22]{ApOn}, \cite[p258]{Magnus1},
and was also proved in \cite{HLLN}. 

\begin{lemma}\label{lem:top}
Let $w=(u,v)\in W$. 
Let $R$ be any polynomial ring over $K$, $f\in R$ be a $w$-homogeneous polynomial of degree $d_f>0$, and $g$ be a nonzero $w$-homogeneous function of degree $d_g\in\mathbb{Z}$ in the fractional field of $R$ such that the Jacobian $[g,f]=0$. 
Define $r\in\mathbb{Z}_{>0}$ to be the largest integer such that $h=f^{1/r}$ is a polynomial.
Then there exists a unique $c\in K\setminus\{0\}$ so that $g=c\cdot h^s$, where $s=rd_g/d_f$ is an integer. 
\end{lemma}

\begin{lemma}\label{lem:Jac}
Suppose $F,$ is a nonzero polynomial and $d=\deg(F_+)$ and $e=\deg(G_+)$.
	Write the homogeneous  decompositions $F=\sum_{i=0}^{d} F_i$, and define
$\widetilde{F}=\ff_d+\ff_{d-1}t+\cdots +\ff_0 t^d$. 
Then $[\widetilde{F},\widetilde{F}^{\frac{e-\gamma}{d}}]=0$
whenever $\frac{e-\gamma}{d} = \ell \cdot \frac{1}{r}$ for $\ell\in \Z$.
\end{lemma}
\begin{proof}
There exists $h=x^{m/r}y^{n/r}+h_1(x,y)t+h_2(x,y)t^2+\cdots\in\mathbb{C}[x^\pm1,y^\pm1][[t]]$ such that $\widetilde{F}=h^r$, $\widetilde{F}^{\frac{e-\gamma}{d}}=h^{\ell}$. 
Then 
$[\widetilde{F},\widetilde{F}^{\frac{e-\gamma}{d}}]=[h^r,h^\ell]=\begin{vmatrix} rh^{r-1}h_x&rh^{r-1}h_y\\ \ell h^{\ell-1}h_x&\ell h^{\ell-1}h_y\end{vmatrix}=r
	\ell h^{r+\ell-2}h_xh_y-r \ell h^{r+\ell-2}h_xh_y=0$.
\end{proof}

\begin{remark}\label{rem:degree}
	Note that when 
	$\frac{e-\gamma}{d} = \ell \cdot \frac{1}{r}$  for some integer $\ell$,
		it follows from \eqref{eq:xA} that 
        $[\widetilde{F}^{\frac{e-\gamma}{d}}]_{t^{\mu-\gamma}}$ is 
		a homogeneous polynomial (respectively,
		Laurent polynomial) of degree $e-\mu$
		when $e-\mu \geq 0$ (respectively, $e-\mu < 0$). 
\end{remark}

\begin{proof}[Proof of Theorem \ref{prop:Magnus_formula_extended}]
Let $\mu_0=d+e-1$.
Define constants 
$$\lambda_\ell=
\begin{cases}
&[F,G],\text{ for $\ell=\mu_0$,}\\ 
&0, \text{ for $\ell \neq \mu_0$}. 
\end{cases}
$$
Then
	\begin{equation}\label{eq:lambda}
		\lambda_\ell=\big[[\widetilde{F},\widetilde{G}]\big]_{t^\ell}
=\big[[\sum_{i\ge0} F_{d-i}t^{i},\sum_{j\ge0} G_{e-j}t^{j}]\big]_{t^\ell}
=\sum_{\stackrel{i,j\ge0}{i+j=\ell}}[F_{d-i},G_{e-j}],
	\end{equation}
so
\begin{equation}\label{eq:GF=FG:variation}
[G_{e-\ell}, F_d]=-[F_d,G_{e-\ell}]=-\lambda_\ell+\sum_{\stackrel{i>0,j\ge0}{i+j=\ell}}[F_{d-i},G_{e-j}].
\end{equation}

We will prove the theorem by induction on $\mu$. 
The base case ($\mu=0$) follows from $[\ff_d, \GG_e]=0$ and \Cref{lem:top}.
For the inductive step, assume that $\mu>0$,
 that $c_0,\dots,c_{\mu-1}$ have already been determined,
and that \eqref{Magnus_formula_equivalence} holds for $\mu'<\mu$.
Note that in the right-hand side of 
\eqref{Magnus_formula_equivalence},
the only coefficients
$c_{\gamma}$ that contribute to the coefficient of $t^{\mu-\gamma}$ are those with 
$\gamma \leq \mu$, so we 
 show that 
requiring 
\eqref{Magnus_formula_equivalence} to hold uniquely determines $c_{\mu}$
	from $G_{e-\mu}$ together with $c_0,\dots , c_{\mu-1}$.
Define 	
	\begin{equation} \label{eq:H}
	  H_{e-\mu}:= G_{e-\mu}-\sum_{\gamma=0}^{\mu-1} c_\gamma [\widetilde{F}^{\frac{e-\gamma}{d}}]_{t^{\mu-\gamma}}.
	\end{equation}
We claim that:
$$\text{ For }\mu\le \mu_0, \quad [H_{e-\mu},F_d]=-\lambda_\mu.$$
Indeed,
\begingroup
\allowdisplaybreaks
\begin{align*}
[H_{e-\mu},F_d]
&=[G_{e-\mu},F_d]-[\sum_{\gamma=0}^{\mu-1} c_\gamma [\widetilde{F}^{\frac{e-\gamma}{d}}]_{t^{\mu-\gamma}},F_d]\\
&\stackrel{(a)}{=}
-\lambda_\mu+\sum_{\stackrel{i>0,j\ge0}{i+j=\mu}}[F_{d-i},G_{e-j}]-\sum_{\gamma=0}^{\mu-1} [c_\gamma [\widetilde{F}^{\frac{e-\gamma}{d}}]_{t^{\mu-\gamma}},F_d]\\
&\stackrel{(b)}{=}-\lambda_\mu+\sum_{j=0}^{\mu-1}[F_{d-\mu+j},\sum_{\gamma=0}^j c_\gamma [\widetilde{F}^{\frac{e-\gamma}{d}}]_{t^{j-\gamma}}]-\sum_{\gamma=0}^{\mu-1} [c_\gamma [\widetilde{F}^{\frac{e-\gamma}{d}}]_{t^{\mu-\gamma}},F_d]\\
&=-\lambda_\mu+
 \sum_{\gamma=0}^{\mu-1}c_\gamma 
\Big(
\sum_{j=\gamma}^{\mu-1}\big[F_{d-\mu+j},[\widetilde{F}^{\frac{e-\gamma}{d}}]_{t^{j-\gamma}}\big]
-\big[[\widetilde{F}^{\frac{e-\gamma}{d}}]_{t^{\mu-\gamma}},F_d \big]
\Big)
\\
&=-\lambda_\mu + \sum_{\gamma=0}^{\mu-1}c_\gamma 
\Big(
\sum_{j=\gamma}^{\mu}\big[F_{d-\mu+j},[\widetilde{F}^{\frac{e-\gamma}{d}}]_{t^{j-\gamma}}\big]
\Big)
\\
&\stackrel{(c)}{=}-\lambda_\mu+ \sum_{\gamma=0}^{\mu-1}c_\gamma 
\Big(
\sum_{i\le \mu-\gamma}\big[[\widetilde{F}]_{t^i},[\widetilde{F}^{\frac{e-\gamma}{d}}]_{t^{\mu-\gamma-i}}\big]
\Big)
\\
&=-\lambda_\mu+\sum_{\gamma=0}^{\mu-1}c_\gamma 
\Big(
\big[[\widetilde{F},\widetilde{F}^{\frac{e-\gamma}{d}}]\big]_{t^{\mu-\gamma}}
\Big)
\stackrel{(d)}{=}-\lambda_\mu
\end{align*}
\endgroup
where (a) uses \eqref{eq:GF=FG:variation}, (b) uses the inductive hypothesis, (c) is obtained by substitution $i=\mu-j$ and the fact that $[\widetilde{F}]_{t^i}=0$ for $i<0$, and (d) uses \Cref{lem:Jac}.

First, consider the case $\mu<\mu_0$.
The above claim asserts that $[H_{e-\mu},F_d]=-\lambda_\mu=0$. 
If $H_{e-\mu}=0$, 
	then $c_{\mu}=0$ is the unique constant satisfying \eqref{Magnus_formula_equivalence}. 
Now suppose $H_{e-\mu}\neq0$. 
By the inductive hypothesis and \Cref{rem:degree}, 
	$H_{e-\mu}$ is a homogeneous Laurent polynomial in $\C[x,y,F_d^{-1/r}]$ of degree $e-\mu$. 
By \Cref{lem:top} (with $f=F_d, g=H_{e-\mu}, h=F_d^{1/r}=x$), we have $r(e-\mu)/d\in\mathbb{Z}$ (which is trivially true), and 
there is a unique element $c_\mu\in \mathbb{C} \setminus\{0\}$
such that 
$$H_{e-\mu}=c_\mu F_d^{\deg H_{e-\mu}/\deg F_d}=c_\mu F_d^{\frac{e-\mu}{d}}=c_\mu x^{e-\mu}.$$
This shows the unique existence of the constant $c_{\mu}$ satisfying \eqref{Magnus_formula_equivalence}.  

\medskip

Next, consider the case $\mu=\mu_0$.  
Let $R=\frac{-\lambda_{\mu_0}}{n-m} \cdot \frac{1}{x^{m-1}y^{n-1}}$.
Then $[R,F_d]=-\lambda_{\mu_0}$, thus 
 $[H_{e-\mu_0}-R,F_d] =0$, so there is a constant $c_{\mu_0}$ such that 
 $H_{e-\mu_0}-R=c_{\mu_0}x^{e-\mu_0}=c_{\mu_0} x^{-d+1} =  c_{\mu_0} [\widetilde{F}^{\frac{e-\mu_0}{d}}]_{t^0}$.
 Then 
\begin{equation}\label{eq:star1}
-\sum_{\gamma=0}^{d+e-1} c_\gamma [\widetilde{F}^{\frac{e-\gamma}{d}}]_{t^{\mu_0-\gamma}} 
=G_{e-\mu_0}-\sum_{\gamma=0}^{d+e-1} c_\gamma [\widetilde{F}^{\frac{e-\gamma}{d}}]_{t^{\mu_0-\gamma}} = R
\end{equation}
which proves \eqref{Magnus_formula_extended}. 
\end{proof}

\begin{thebibliography}{99}
\bibitem{A}
 S. S. Abhyankar, Lectures on expansion techniques in algebraic geometry, Tata Institute of Fundamental Research, Bombay, 1977.


\bibitem{AM}
S. S. Abhyankar and T. T. Moh, \emph{Embeddings of the line in the plane}, J. Reine Angew. Math. \textbf{276} (1975), 148--166.


\bibitem{AdEs}
 K. Adjamagbo, A.R.P. van den Essen, \emph{Eulerian operators and the Jacobian conjecture. III}, J. Pure Appl. Algebra \textbf{81} (1992), 111--116.

\bibitem{ApOn}
H. Appellate and H. Onishi, \emph{The Jacobian conjecture in two variables}, J. Pure
Appl. Algebra \textbf{37} (1985), 215--227.

\bibitem{BCW}
H. Bass, E. Connell and D. Wright, \emph{The Jacobian conjecture: reduction of degree and formal expansion of the inverse},  Bull. Amer. Math. Soc. \textbf{7} (1982), 287--330. 

\bibitem{BK}
A. Belov-Kanel and M. Kontsevich, \emph{The Jacobian conjecture is stably equivalent to the Dixmier conjecture}, Mosc. Math. J. \textbf{7} (2007), 209--218.

\bibitem{dBY}
M. de Bondt and D. Yan,  \emph{Irreducibility properties of Keller maps}, Algebra Colloq. \textbf{23} (2016), 663--680.

\bibitem{CA} 
E.~Casas-Alvero, Singularities of plane curves, London Math. Soc. Lecture Note Ser., 276, 
Cambridge University Press, Cambridge, 2000. 

\bibitem{CN} P. Cassou-Nogu\'es, Newton trees at infinity of algebraic curves. Affine
algebraic geometry, 1--19, CRM Proc. Lecture Notes, 54, Amer. Math. Soc.,
Providence, RI, 2011. 

\bibitem{NVC}
N. V. Chau, \emph{Plane Jacobian conjecture for simple polynomials}, Ann. Polon. Math. \textbf{93} (2008),  247--251. 

\bibitem{CW1}
C.C. Cheng and S.S.-S. Wang, \emph{A case of the Jacobian conjecture}, J. Pure Appl. Algebra \textbf{96} (1994), 15--18.

\bibitem{CW2}
C.C. Cheng and S.S.-S. Wang, \emph{Radial similarity of Newton polygons},  Automorphisms of affine spaces (Cura\,cao, 1994), 157--167, Kluwer Acad. Publ., Dordrecht, 1995. 

\bibitem{CZ}
E. Connell and J. Zweibel, \emph{Subgroups of polynomial automorphisms}, Bull. Amer. Math. Soc. (N.S.) \textbf{23} (1990), 401--406.

\bibitem{DEZ}
H. Derksen, A.R.P. van den Essen and Wenhua Zhao, \emph{The Gaussian moments conjecture and the Jacobian conjecture}, Israel Journal of Math. \textbf{219} (2017), 917--928.

\bibitem{Dru}
L.M. Dru\.zkowski, \emph{An effective approach to Keller's Jacobian conjecture}, Math. Ann. \textbf{264} (1983), 303--313.

\bibitem{EssenTutaj}
A.R.P.  van den Essen and H. Tutaj, \emph{A remark on the two-dimensional Jacobian conjecture}, J. Pure Appl. Algebra \textbf{96} (1994), 19--22.
 

\bibitem{Essen}
A.R.P. van den Essen, \emph{Polynomial automorphisms and the Jacobian conjecture}, Alg\`ebre non commutative, groupes quantiques et invariants (Reims, 1995), 55--81, S\'emin. Congr., 2, Soc. Math. France, Paris, 1997.
 
 \bibitem{vdEssen}
A. van den Essen, Polynomial automorphisms and the Jacobian conjecture,
Progress in Mathematics, 190.
Birkh\"{a}user Verlag, Basel, 2000. xviii+329 pp. ISBN 3-7643-6350-9.

 
 \bibitem{EssenWZ}
A.R.P. van den Essen, D. Wright and W. Zhao, \emph{On the image conjecture}, J. Algebra \textbf{340} (2011), 211--224.
 
 
\bibitem{GHLL2}
J. Glidewell, W. E. Hurst, K. Lee, L. Li, \emph{On the two-dimensional Jacobian conjecture: Magnus' formula revisited, II}, arXiv:2205.12792.
 
 \bibitem{GHLL3}
J. Glidewell, W. E. Hurst, K. Lee and L. Li, \emph{On the two-dimensional Jacobian conjecture: Magnus' formula revisited, III},  Contemporary Mathematics, Volume 791, 2024.

 \bibitem{GGV14} J. A. Guccione, J. J. Guccione and  C. Valqui, \emph{A system of polynomial equations related to the 
Jacobian conjecture},  arXiv:1406.0886.

 
 \bibitem{Gwo}
 J. Gwo\'zdziewicz, \emph{Injectivity on one line}, Bull. Soc. Sci. \begin{picture}(1,1)\put(1,3){\line(1, 0){3}}\end{picture}L\'od\'z 7 (1993), 59--60, S\'erie: Recherches sur les d\'eformationes XV.
 
 \bibitem{H} 
R. Heitmann, \emph{On the Jacobian conjecture}, J. Pure Appl. Algebra \textbf{64} (1990), 35--72.

\bibitem{HLLN}
W. E. Hurst, K. Lee, L. Li and G. D. Nasr, \emph{On the two-dimensional Jacobian conjecture: Magnus' formula revisited, I}, Rocky Mountain Journal of Mathematics 53 (2023), no. 3, 791--806.


\bibitem{Hub}
E.-M.G.M. Hubbers, \emph{The Jacobian Conjecture: Cubic Homogeneous Maps in Dimension Four}, Master's thesis, University of Nijmegen, 1994, directed by A.R.P. van den Essen.

\bibitem{JZ}
P. J\polhk edrzejewicz and  J. Zieli\'{n}ski,  \emph{An approach to the Jacobian conjecture in terms of irreducibility and square-freeness}, Eur. J. Math. \textbf{3} (2017), 199--207.


\bibitem{Keller}
O. H. Keller, \emph{Ganze Cremona-Transformationen}, Monats. Math. Physik \textbf{47} (1939), 299--306.

\bibitem{Kire}
M. Kirezci, \emph{The Jacobian conjecture. I and II}, \.Istanbul Tek. \"Univ. B\"ul. \textbf{43} (1990), 421--436 and 451--457. 

\bibitem{L} 
J. Lang, \emph{Jacobian pairs II}, J. Pure Appl. Algebra \textbf{74} (1991), 61--71.

\bibitem{LM}
J. Lang and S. Maslamani, \emph{Some results on the Jacobian conjecture in higher dimension}, J. Pure Appl. Algebra \textbf{94} (1994), 327--330.
 
\bibitem{Magnus1}
A. Magnus, \emph{Volume preserving transformations in several complex variables}, Proc. Amer. Math. Soc. \textbf{5} (1954), 256--266.


\bibitem{ML1} 
L. Makar-Limanov, \emph{On the Newton polygon of a Jacobian mate}, Automorphisms in birational and affine geometry, 469--476, Springer Proc. Math. Stat., 79, Springer, Cham, 2014.

\bibitem{M}
L. Makar-Limanov, \emph{On the Newton polytope of a Jacobian pair}, Izvestiya: Mathematics \textbf{85} (2021), 457--467.

\bibitem{MU}
L. Makar-Limanov and U. Umirbaev, \emph{The Freiheitssatz for Poisson algebras}, J. Algebra \textbf{328} (2011), 495--503. 

\bibitem{MW} J. McKay and S. Wang, \emph{A note on the Jacobian condition and two points
at infinity}, Proc. Amer. Math. Soc. \textbf{111} (1991), 35--43.

\bibitem{MO}
G. Meisters and C. Olech, \emph{Power-exact, nilpotent, homogeneous matrices}, 
Linear and Multilinear Algebra \textbf{35} (1993), 225--236. 

\bibitem{Moh}
T.T. Moh, \emph{On the Jacobian conjecture and the configurations of roots},  
J. Reine Angew. Math. \textbf{340} (1983), 140--212. 

\bibitem{Na1} 
M. Nagata, \emph{Two-dimensional Jacobian conjecture}, Algebra and topology (Taejon, 1988), 77--98, Korea Inst. Tech., Taejon, 1988.

\bibitem{Nagata}
M. Nagata, \emph{Some remarks on the two-dimensional Jacobian conjecture}, Chin. J. Math. \textbf{17} (1989), 1--7.

\bibitem{NaBa}
Y. Nakai and K. Baba, \emph{A generalization of Magnus' theorem}, Osaka J. Math. \textbf{14}
(1977), 403--409.

\bibitem{N0}
	A. Nowicki, \emph{On the Jacobian equation $J(f,g)=0$ for polynomials in 
	$k[x,y]$}, Nagoya Math J, 1988, 151--157.

\bibitem{NN1} 
A. Nowicki, Y. Nakai, \emph{On Appelgate--Onishi’s lemmas}, J. Pure Appl. Algebra \textbf{51} (1988),  305--310.

\bibitem{NN2} 
A. Nowicki, Y. Nakai, \emph{Correction to: ``On Appelgate--Onishi’s lemmas”}, 
J. Pure Appl. Algebra \textbf{58} (1989), 101--101.

\bibitem{Ok} 
M. Oka, \emph{On the boundary obstructions to the Jacobian problem}, Kodai
Math. J. \textbf{6} (1983), 419--433.

\bibitem{Re}
R. Rentschler, \emph{Op\'{e}rations du groupe additif sur le plane affine} (French), 
C. R. Acad. Sci. Paris S\'{e}r. A-B \textbf{267} (1968), A384--A387.

\bibitem{Rosenlicht}
A. Bialynick-Birula and M. Rosenlicht,
\emph{Injective morphisms of real algebraic varieties},
Proc. Amer. Math. Soc. \textbf{13} (1962), 200--203.


\bibitem{Silverman}
J. H. Silverman, \emph{The arithmetic of elliptic curves}, Graduate Texts in Mathematics 106, second edition, Springer 2009.

 \bibitem{GGV} 
 C. Valqui, J. A. Guccione and J. J. Guccione, \emph{On the shape of possible counterexamples to the Jacobian conjecture}, J. Algebra \textbf{471} (2017), 13--74.

\bibitem{Wang}
S.S.-S. Wang, \emph{A Jacobian criterion for separability}, J. Algebra \textbf{65} (1980), 453--494.

\bibitem{Yag}
A.V. Yagzhev, \emph{On Keller's problem}, Siberian Math. J. \textbf{21} (1980), 747--754.

\bibitem{Yu}
J.-T. Yu, \emph{On the Jacobian Conjecture: reduction of coefficients}, J. Algebra
\textbf{171} (1995), 515--523.


\end{thebibliography}
\end{document}